\documentclass[leqno,a4paper]{amsart}
\usepackage{amsmath,amsthm,amssymb}
\usepackage[T1]{fontenc}

\newtheorem{Th}{Theorem}
\newtheorem{Prop}[Th]{Proposition}
\newtheorem{Lemma}[Th]{Lemma}

\theoremstyle{remark}
\newtheorem{Remark}{Remark}

\newtheorem{Problem}{Problem}
\newtheorem*{Notation}{Notation}

\newcommand{\bez}{\nopagebreak\hspace*{\fill}\nolinebreak$\Box$}

\newcommand{\al}{\alpha}

\newcommand{\Q}{\mathbb{Q}}
\newcommand{\R}{{\mathbb{R}}}
\newcommand{\T}{{\mathbb{T}}}
\newcommand{\C}{{\mathbb{C}}}
\newcommand{\Z}{{\mathbb{Z}}}
\newcommand{\N}{{\mathbb{N}}}
\newcommand{\s}{{\mathbb{S}}}
\newcommand{\D}{{\mathbb{D}}}

\newcommand{\ov}{\overline}
\newcommand{\vep}{\varepsilon}
\newcommand{\va}{\varphi}

\newcommand{\var}{\operatorname{Var }}

\begin{document}
\title[On Hausdorff
dimension of the set of closed orbits]{On Hausdorff dimension of
the set of closed orbits for a cylindrical transformation}
\author{K. Fr\k{a}czek
\and M. Lema\'nczyk}
\address{K.\ Fr\k{a}czek \\ Faculty of Mathematics
and Computer Science\\ Nicolaus Copernicus University\\ ul.
Chopina 12/18, 87-100 Toru\'n, Poland}
\address{M.\ Lema\'nczyk\\
Institute of Mathematics of Polish Academy of Sciences, ul.\
\'Sniadeckich 8, 00-950 Warszawa, Poland and Faculty of
Mathematics and Computer Science\\ Nicolaus Copernicus
University\\ ul. Chopina 12/18, 87-100 Toru\'n, Poland}

\email{fraczek@mat.umk.pl, mlem@mat.umk.pl}
\thanks{Research partially supported by Polish MNiSzW grant
N N201 384834; partially supported by Marie Curie ``Transfer of
Knowledge'' EU program -- project MTKD-CT-2005-030042 (TODEQ)}


\begin{abstract}We deal with Besicovitch's problem of existence
of discrete orbits for transitive cylindrical transformations
$T_\varphi:(x,t)\mapsto(x+\alpha,t+\varphi(x))$ where
$Tx=x+\alpha$ is an irrational rotation on the circle $\T$ and
$\varphi:\T\to\R$ is continuous, i.e.\ we try to estimate how big
can be the set
$D(\alpha,\varphi):=\{x\in\T:|\varphi^{(n)}(x)|\to+\infty\text{ as
}|n|\to+\infty\}$. We show that for almost every $\alpha$ there
exists $\varphi$ such that the Hausdorff dimension of
$D(\alpha,\varphi)$ is at least $1/2$. We also provide a
Diophantine condition on $\alpha$ that guarantees the existence of
$\varphi$ such that the dimension of $D(\alpha,\varphi)$ is
positive. Finally, for some multidimensional rotations $T$ on
$\T^d$, $d\geq3$, we construct smooth $\varphi$ so that the
Hausdorff dimension of $D(\alpha,\varphi)$ is positive.
\end{abstract}

\maketitle

\section{Introduction}
Let $T:X\to X$ be a minimal homeomorphism of a compact metric
space $(X,d)$ and let $\va:X \to\R$ be a continuous function.
Denote by $T_\va:X\times\R\to X\times\R$ the corresponding
cylindrical transformation
\[T_\va(x,t)=(Tx,t+\va(x)).\]
Then $T^n_\va(x,t)=(T^nx,t+\va^{(n)}(x))$ for every integer $n$,
where
\[\va^{(n)}(x)=\left\{
\begin{array}{ccl}
\va(x)+\va(Tx)+\ldots+\va(T^{n-1}x)&\text{ if }&n>0\\
 0&\text{
if }&n=0\\
-(\va(T^{-1}x)+\ldots+\va(T^{n+1}x)+\va(T^{n}x))&\text{ if }&n<0.
\end{array}
\right.\] Cylindrical transformations appear naturally when
studying some autonomous ordinary differential equations on $\R^3$
 or on other non-compact
manifolds (cf.\ \cite{Po} and Section~\ref{difeq}). Moreover, in
the case of circle rotations $T$, cylindrical transformations
yield a broad and interesting class of homeomorphisms of the
plane. If $Tx=x+\alpha$ then every continuous function
$\va:\T\to\R$ defines the homeomorphism
$f_{\alpha,\va}:\R^2\to\R^2$, $f_{\alpha,\va}(re^{2\pi
i\omega})=re^{\varphi(\omega)+2\pi i(\omega+\alpha)}$ for $r\geq
0$ and $\omega\in\T$. The homeomorphism $f_{\alpha,\va}$ has a
fixed point at zero and restricted to $\R^2\setminus\{(0,0)\}$ is
topologically isomorphic to $T_\va$ via the map
\[\T\times\R\ni(\omega,r)\mapsto e^re^{2\pi i \omega}\in\R^2\setminus\{(0,0)\}.\]

A surprising property of  general cylindrical transformations is
that they are never minimal, that is, there are points whose
orbits are not dense \cite{Be}, \cite{Me-Si},  (see also
\cite{LeC-Yo} for a more general non-minimality result in case of
homeomorphisms of the cylinder $\T\times\R$) and the problem of
classifying minimal subsets for such transformations is still
open\footnote{ The situation changes if instead of minimality we
consider so called positive minimality i.e.\ for a continuous map
$T$ (not necessarily invertible) of a locally compact space $X$ we
require all semi-orbits $\{T^nx:\:n\geq0\}$, $x\in X$, to be dense
in $X$. As it has already been noticed in \cite{Aus} (see Chapter
I, Exercise 11 or the subsequent article \cite{Ber}) if there is a
recurrent point in $X$ and $T$ is positively minimal then $X$ has
to be compact. Take now an arbitrary continuous map $T$ of a
locally compact space $X$ and suppose that $M\subset X$ is
positively minimal. Then $M$ is locally compact and it follows
that either $M$ is a discrete orbit or $M$ is compact. Therefore
there are no positively minimal subsets for transitive $T_\varphi$
as above  -- if $M\subset X\times\R$ is positively minimal then it
cannot be  a discrete orbit and  if $M$ is compact then by
\cite{Go-He}, $\varphi$ is a coboundary and therefore $T_\varphi$
is not transitive.}. Clearly, a minimal subset arises if we are
given a discrete orbit. Recall that a subset $S$ of a topological
space is discrete if every point $x\in S$ has a neighborhood $U$
such that $S\cap U=\{x\}$. Moreover, the orbit of $(x,t)\in
X\times\R$ for the cylindrical transformation $T_\va$ is discrete
if and only if
\[|\va^{(n)}(x)|\to+\infty\text{ as }|n|\to+\infty.\]
If $T$ is uniquely ergodic with $\mu$ the only $T$--invariant
measure and if $\int_X\va\,d\mu\neq0$ then by unique ergodicity,
$\va^{(n)}/n\to \int \va\,d\mu$ uniformly as $|n|\to+\infty$.
Therefore $|\va^{(n)}(x)|\to+\infty$ as $|n|\to+\infty$ for each
$x\in X$, i.e.\ $\va$ is transient, and hence every orbit of
$T_{\va}$ is discrete. It follows that the partition of $X\times
\R$ into orbits of $T_{\varphi}$ yields the decomposition into
minimal components. Yet, in one more situation $X\times\R$ is the
union of minimal components -- it is the case when
$\int_X\va\,d\mu=0$ and $\va(x)=j(x)-j(Tx)$  for a continuous
function $j:X\to\R$, i.e.\ when $\va$ is a coboundary; indeed, the
minimal components are of the form $\{(x,j(x)+a):\:x\in X\}$,
$a\in\R$. Clearly, in this case there are no discrete orbits, in
fact, $j$ exists if and only if each orbit of $T_\varphi$ is
bounded \cite{Go-He}.

When we restrict our considerations to $T$ which is a minimal
rotation on a compact metric group (the case which is well known
to be uniquely ergodic) then we have the following.
\begin{Prop}[see \cite{At}, \cite{Go-He} or \cite{Le-Me}]\label{dych}
If $\varphi$ is not transient nor $\va$ is a coboundary then
$T_\va$ has a dense orbit, i.e.\ $T_\va$ is topologically
transitive.
\end{Prop}
\noindent Note that, by Proposition~\ref{dych}, it follows that if
$T$ is a minimal rotation, $\va$ has zero mean and $T_\va$ has a
discrete orbit then $T_\va$ is automatically topologically
transitive.

From now on we will only deal with the transitive case and we
assume that $T$ is a minimal rotation on $X$. In this case the set
of transitive points is $G_\delta$ and dense, however it is always
a proper subset of $X\times\R$ since $T_\varphi$ is not minimal.
This set is usually also large from the measure-theoretic point of
view; indeed, if we assume ergodicity of $T_\va$ (with respect to
the product of Haar measure on $X$ and Lebesgue measure on $\R$)
then each open subset of $X\times\R$ has positive measure and
since $X\times\R$ is second countable, the complement of the set
of transitive points has measure zero. Even in case of $T$ a
minimal rotation, the problem of classifying possible minimal
subsets for the corresponding cylindrical transformations remains
open. It is even open in case of irrational rotations on the
circle, although in the latter case we would like to emphasize
that if $\va$ is too smooth then there are no minimal subsets at
all. More precisely, if $\va:\T\to\R$ is of bounded variation then
$T_\varphi$ has no minimal subset (see \cite{Ma-Sh} and
\cite{Me-Si}). However, Besicovitch \cite{Be} already in 1951
showed that, if we require $\va$ to be only continuous, then for
$T_\varphi$ a minimal subset can exist, namely, despite its
topological transitivity $T_\varphi$ can have a discrete orbit.
The problem of coexistence of dense orbits (topological
transitivity) and discrete orbits for cylindrical transformations
when $T$ is a minimal rotation is called the Besicovitch problem,
and we will call the cylindrical transformation $T_\va$
Besicovitch if indeed dense and discrete orbits for $T_\va$
coexist. In the present paper we will deal with the Besicovitch
problem for rotations on finite dimensional tori $\T^d$.

In Section~\ref{secpodst} we show that the phenomenon discovered
by Besicovitch \cite{Be} -- for a particular irrational $\alpha$
there exists a continuous $\va:\T\to\R$ such that $T_\va$ is
transitive and admits discrete orbits -- in fact happens for each
irrational $\alpha$. In other words, we show that for every
irrational $\alpha\in\T$ there exists a continuous $\va:\T\to\R$
such that $T_\va$ is Besicovitch; note that this implies the
existence of Besicovitch cylindrical transformations over each
minimal rotation on $\T^d$, $d\geq2$. Indeed, if $R$ is a rotation
on $\T^{d-1}$ such that $T\times R$ is minimal then the
cylindrical transformation $(T\times R)_{\bar{\va}}$, with
$\bar{\va}(x_1,\ldots,x_d)=\va(x_1)$, is Besicovitch if $T_\va$ is
Besicovitch  (this follows from Proposition~\ref{dych} since
$\widetilde{\varphi}$ is not transient and since no orbit of
$(T\times R)_{\widetilde{\varphi}}$ is bounded,
$\widetilde{\varphi}$ is not a coboundary).

In Section~\ref{holder}, for $\alpha$ satisfying some Diophantine
conditions, our construction of $\va$ is improved and we obtain
$\gamma$--H\"older continuous functions $\va$ such that $T_\va$ is
Besicovitch (it turns out however that in all our constructions
$\gamma<1/2$). A slight modification of the construction yields
$\va:\T\to\R$ whose Fourier coefficients are $\mbox{O}(\log
|n|/|n|)$; see Section~\ref{oduze}. We have already mentioned that
in case $\va$ is of bounded variation, $T_\va$ is not Besicovitch.
This is one more (direct) consequence of the classical
Denjoy-Koksma inequality (see e.g.\ \cite{Ku-Ni} or \cite{Her});
indeed, the inequality
\begin{equation}\label{Dkoksma}
|\va^{(q_n)}(x)|\leq\operatorname{Var}\,\varphi\;\;\mbox{ for each
$x\in\T$},\end{equation} where $(q_n)$ is the sequence of
denominators of $\alpha$ means in particular that the orbit of
each point $(x,t)$ is not discrete (in fact $T_\va$ has no minimal
subsets at all; see \cite{Ma-Sh} and \cite{Me-Si}). In
\cite{Aa-Le-Ma-Na}, an inequality similar to~(\ref{Dkoksma}) has
been proved in $L^2$ for functions whose Fourier coefficients are
$\mbox{O}(1/|n|)$, we have been however unable to decide whether
there exists a continuous $\va:\T\to\R$ whose Fourier coefficients
are $\mbox{O}(1/|n|)$ and which is Besicovitch for some rotation
$Tx=x+\alpha$. Recently, in \cite{Kw-Si}, for every minimal
odometer $T$ the existence of Besicovitch cylindrical
transformation  has been proved to exist.

We then pass to deal with the Besicovitch problem for minimal
rotations on higher dimensional tori $\T^d$, $d\geq2$. As it was
shown by Yoccoz in \cite{Yo}, the Denjoy-Koksma inequality does
not hold anymore in higher dimensions and one can expect that
among smooth cylindrical transformations there are Besicovitch
cylindrical maps. Such are indeed shown to exist in
Section~\ref{highdim}. More precisely, we prove that for every
$r\geq 1$ there exist $d\geq 3$, a minimal rotation
$T:\T^d\to\T^d$ and $\va:\T^d\to\R$ of class $C^r$ such that
$T_\va$ is Besicovitch; the construction is based on Yoccoz's
method from \cite{Yo}.

Once we know that Besicovitch cylindrical transformations exist
for each minimal rotation on $\T^d$, another natural problem
arises to discuss the size of the set of points whose orbits are
discrete. More precisely, we will deal with the set
\[D(\alpha,\va)=\{x\in\T^d:\lim_{n\to\pm\infty}|\va^{(n)}(x)|\to+\infty\}.\]
By our standing assumption of transitivity, $\int\va(x)\,dx=0$ and
thus $T_\va$ is recurrent as an infinite measure-preserving system
(see \cite{Aa}, \cite{Sch}), so for a.e.\ $x\in \T^d$ there exists
$k_n=k_n(x)\to+\infty$ such that $\va^{(k_n)}(x)\to 0$, hence
$D(\alpha,\va)$ has zero Lebesgue measure. Moreover,  the set of
transitive points for $T_\va$ is $G_\delta$ dense, so the set of
points whose orbits are discrete is a first category set.
Furthermore, this set is equal to $D(\alpha,\va)\times\R$, so
$D(\alpha,\va)$ is a first category subset of $\T^d$ (which is
dense if it is nonempty). Consequently, $D(\alpha,\va)$ is small
from both the topological and the measure theoretical point of
view. We are interested in the Hausdorff dimension of
$D(\alpha,\va)$.

If $d=1$ then for almost every $\alpha\in\T$, using a modification
of the construction from Section~\ref{secpodst}, we built a
continuous function $\va:\T\to\R$ with zero mean such that
$\dim_HD(\alpha,\va)>0$. Moreover, we give a lower bound on the
Hausdorff dimension related to some Diophantine condition of
$\alpha$; see Section~\ref{secdimension}. We also study the
coexistence problem of discrete orbits of different types, more
precisely, the size of sets
\[D^{s_-s_+}(\alpha,\va)=\{x\in\T:\lim_{m\to-\infty}\va^{(m)}(x)\to
s_-\infty\text{ and } \lim_{m\to+\infty}\va^{(m)}(x)\to
s_+\infty\}\] for $s_-,s_+\in\{-,+\}$ is investigated. We show
that the coexistence of all types of discrete orbits appears for
some (transitive) cylindrical transformations; we mention that the
same phenomenon was also observed for cylindrical transformations
over odometers in \cite{Kw-Si}. For almost every $\alpha\in\T$ we
construct a class of examples for which
$\dim_HD^{s_-s_+}(\alpha,\va)\geq 1/2$ for every pair $(s_-,s_+)$.
This gives evidence that for transitive  homeomorphisms of the
plane the coexistence of orbits with completely different behavior
is possible. Indeed, returning to the homeomorphisms of the plane
mentioned at the beginning of this section, let us consider
$f_{\alpha,\va}:\R^2\to\R^2$. Then for each $\omega\in\T$ points
from the ray $\operatorname{Ray}(\omega)=\{re^{i\omega}:r>0\}$
generate orbits of the same type and
\begin{itemize}
\item if $\omega\in D^{--}(\alpha,\va)$ then each
$x\in\operatorname{Ray}(\omega)$ generates a homoclinic orbit
attracted by zero;
\item if $\omega\in D^{++}(\alpha,\va)$ then each
$x\in\operatorname{Ray}(\omega)$ generates  a discrete orbit;
\item if $\omega\in D^{-+}(\alpha,\va) (D^{+-}(\alpha,\va))$ then  for each
$x\in\operatorname{Ray}(\omega)$ one semi-orbit is attracted by
zero and  another semiorbit escapes to the infinity.
\end{itemize}
Therefore  there exists a transitive homeomorphism
$f_{\alpha,\va}:\R^2\to\R^2$ such that  each of the above orbit
type appears and the Hausdorff dimension of the set of the
corresponding points is no smaller than $3/2$.

In the higher dimensional case, for every $r\geq 1$ we construct
$\alpha\in\T^d$ and a $C^r$--function $\va:\T^d\to\R$ with zero
mean such that $\dim_HD^{++}(\alpha,\va)>0$; see
Section~\ref{highdim}.

As an application, in Section~\ref{difeq} we demonstrate a family
of continuous (or even H\"older) perturbations of some integrable
systems which completely destroys its integrable dynamics. More
precisely, the perturbed systems have plenty of orbits which are
dense, homoclinic and heteroclinic to limit cycles.

The authors are grateful to the referee for pointing out the
references \cite{Ber} and \cite{Ma-Sh}.

\section{Construction}\label{secpodst}
By $\T$ we will mean the group $\R/\Z$ which most of time will be
treated as $[0,1)$ with addition mod $1$. By $\{t\}$ we denote the
fractional part of $t$ and $\|t\|$ is the distance of $t$ from the
set of integers. Denote by $\lfloor t\rfloor$ and $\lceil t\rceil$
the floor and the ceiling of $t$ respectively.

We will show that for each irrational $\al\in\T$ we can construct
a continuous $\va:\T\to\R$ so that the corresponding cylindrical
flow $T_{\va}$ is Besicovitch, i.e.\ it is topologically
transitive but it has some minimal orbits. By
Proposition~\ref{dych}, we only need to construct a continuous
function $\va$ with integral zero and such that
$$
\va^{(n)}(0)\to+\infty\;\;\mbox{when}\;\;|n|\to+\infty;$$ indeed
such a $\va$ is neither a coboundary nor transient, so $T_\va$
must be topologically transitive.

Fix an irrational $\al\in[0,1)$. Let $(p_n/q_n)$ be the sequence
of convergents of $\al$, i.e.
\begin{eqnarray*}
q_{-1}=0,& q_0=1,& q_{n}=a_nq_{n-1}+q_{n-2}\text{ for }n\geq 1\\
p_{-1}=1,& p_0=0,& p_{n}=a_np_{n-1}+p_{n-2}\text{ for }n\geq 1,
\end{eqnarray*} where $[0;a_1,a_2,\ldots]$ is the continued
fraction of $\alpha$. We have (see e.g.\ \cite{Kh})
\begin{equation}\label{m1}
\frac1{2q_nq_{n+1}}
<(-1)^n\left(\alpha-\frac{p_n}{q_n}\right)=\left|\alpha-\frac{p_n}{q_n}\right|<
\frac1{q_nq_{n+1}},
\end{equation}
hence
\begin{equation}\label{mip}
\frac1{2q_{n+1}} <\|q_n\alpha\|<\frac1{q_{n+1}}.
\end{equation}
 Let $(M_n)$ be a sequence of natural numbers such that
\begin{equation}\label{m2}M_n\to+\infty\end{equation}
and
\begin{equation}\label{m3}\sum_{n=1}^\infty\frac{M_n}{q_{n-1}}<+\infty.\end{equation}
Set
\begin{equation}\label{m4}L_n=M_n\cdot\frac{q_nq_{n+1}}{q_{n-1}}.\end{equation}
In view of~(\ref{m4}) and~(\ref{m3}),
\begin{equation}\label{m5}
\sum_{n=1}^\infty \frac{L_n}{q_nq_{n+1}}<+\infty.
\end{equation}

We now define $f_n:[0,1)\to\R^+$ which is Lipschitz continuous
(with the Lipschitz constant equal to $L_n$), $1/q_n$--periodic,
$f_n(0)=0$ and $|f_n'(x)|=L_n$ for all $x\in[0,1)$ except for the
integer multiples of $\frac{1}{2q_n}$. Notice that $f_n(y)=L_ny$
for $y\in[0,1/(2q_n))$ and $f_n(y)=L_n(1/q_n-y)$ for
$y\in[1/(2q_n),1/q_n]$. Using $1/q_n$--periodicity of $f_n$
and~(\ref{m1}) for each $x\in[0,1)$ we have
$$
|f_n(x+\al)-f_n(x)|=\left|f_n(x+\al)-f_n(x+\frac{p_n}{q_n})\right|\leq
L_n\left|\al-\frac{p_n}{q_n}\right|<\frac{L_n}{q_nq_{n+1}},$$ so
\begin{equation}\label{normsup}
\|f_n(\,\cdot\,+\alpha)-f_n(\,\cdot\,)\|_{C(\T)}<\frac{L_n}{q_nq_{n+1}}
\end{equation}
and it follows from~(\ref{m5}) that the series
$$
\va(x)=\sum_{n=1}^\infty(f_n(x+\alpha)-f_n(x))$$ converges
uniformly, so $\va$ is continuous and clearly
$\int_0^1\va(x)\,dx=0$. For each integer $k$ we  have
$$
\va^{(k)}(x)=\sum_{n=1}^\infty (f_n(x+k\al)-f_n(x)),$$ in
particular
\begin{equation}\label{m8}\va^{(k)}(0)=\sum_{n=1}^\infty
f_n(k\al).\end{equation}

We will show that $\va^{(k)}(0)\to+\infty$ when $|k|\to+\infty$.
Fix a nonzero integer $k$. There is a unique $n=n(k)\geq 0$ such
that $q_n\leq |k|<q_{n+1}$. By~(\ref{m1}) applied to $n+1$ we have
$$\frac{|k|}{2q_{n+1}q_{n+2}}<\left|k\al
-k\frac{p_{n+1}}{q_{n+1}}\right|<\frac
{|k|}{q_{n+1}q_{n+2}}<\frac{1}{q_{n+2}},$$ so
\begin{equation*}\label{m9}
\left|k\al-\frac{kp_{n+1}}{q_{n+1}}\right|>\frac{q_n}{2q_{n+1}q_{n+2}}.
\end{equation*}
Moreover. \begin{eqnarray*}
\left|k\al-\frac{kp_{n+1}}{q_{n+1}}\right|&<&\frac{1}{q_{n+2}}=
\frac1{q_{n+1}}-\left(\frac1{q_{n+1}}-\frac{1}{q_{n+2}}\right)=
\frac1{q_{n+1}}-\frac{q_{n+2}-q_{n+1}}{q_{n+1}q_{n+2}}\\
&\leq&\frac1{q_{n+1}}-\frac{q_n}{q_{n+1}q_{n+2}}<
\frac1{q_{n+1}}-\frac{q_n}{2q_{n+1}q_{n+2}}.
\end{eqnarray*}
Since $f_{n+1}$ is $1/q_{n+1}$--periodic and
$f_{n+1}(-x)=f_{n+1}(x)$,
\[f_{n+1}(k\al)=f_{n+1}\left(k\al-k\frac{p_{n+1}}{q_{n+1}}\right)=
f_{n+1}\left(\left|k\al-k\frac{p_{n+1}}{q_{n+1}}\right|\right).\]
As
\[\frac{q_n}{2q_{n+1}q_{n+2}}<\left|k\al-\frac{kp_{n+1}}
{q_{n+1}}\right|<\frac1{q_{n+1}}-\frac{q_n}{2q_{n+1}q_{n+2}},\]
by the definition of $f_{n+1}$ and (\ref{m4}), \[f_{n+1}(k\al)\geq
f_{n+1}\left(\frac{q_n}{2q_{n+1}q_{n+2}}\right)=
L_{n+1}\cdot\frac{q_n}{2q_{n+1}q_{n+2}}=M_{n+1}/2.\]
Since all functions $f_l$ are nonnegative,  it follows that
$$\va^{(k)}(0)=\sum_{l=1}^\infty f_{l}(k\al)\geq f_{n+1}(k\al)\geq
M_{n+1}/2$$ which tends to $+\infty$ in view of~(\ref{m2}) and of
the fact that $n=n(k)\to+\infty$ when $|k|\to+\infty$.

\section{H\"older continuity condition}\label{holder}
We need the following simple lemma.
\begin{Lemma}\label{kf1} Let $(X,d)$ be a compact metric space.
Let $(w_n)_{n=1}^\infty$ be an increasing sequence of positive
real numbers with $w_n\to+\infty$ such that for every $0<\beta<1$
there exists $D_\beta>0$ for which
\begin{eqnarray}\label{asla}
\sum_{k=1}^nw_k^\beta\leq D_\beta w^\beta_n\;\;\mbox{ and
}\sum_{k=n}^{\infty}\frac{1}{w_k^\beta}\leq\frac{D_\beta}{w_n^{\beta}}\mbox{
for all }n\in\N.
\end{eqnarray}
 Assume that $\va(x)=\sum_{n=1}^\infty\va_n(x)$,
where $\va_n:X\to\R$ is Lipschitz continuous with a Lipschitz
constant $L(\va_n)=L_n$ such that for some $0<\gamma<1$ we have
\[L_n\leq w_n^{1-\gamma} \;\;\text{ and }\;\;\|\va_n\|_{C(X)}\leq
\frac1{w_n^\gamma} \;\;\text{ for }n\geq 1.\] Then $\va:X\to\R$ is
$\gamma$--H\"older continuous.
\end{Lemma}
\begin{proof}
Suppose that $$\frac1{w_{n+1}}<d(x,y)\leq\frac1{w_n}.$$ Then
\begin{eqnarray*}
|\va(x)-\va(y)|&\leq&\sum_{k=1}^n|\va_k(x)-\va_k(y)|+
\sum_{k=n+1}^\infty|\va_k(x)-\va_k(y)|\\
&\leq&\sum_{k=1}^nL_kd(x,y)+2\sum_{k=n+1}^\infty\|\va_k\|_{C(X)}\\
&\leq&
d(x,y)\sum_{k=1}^nw_k^{1-\gamma}+2\sum_{k=n+1}^\infty\frac1{w_k^\gamma}\\
&\leq&
D_{1-\gamma}d(x,y)w_n^{1-\gamma}+2D_{\gamma}\frac1{w_{n+1}^\gamma}\leq
Cd(x,y)^\gamma
\end{eqnarray*}
because $d(x,y)^{1-\gamma}\leq\frac1{w_n^{1-\gamma}}$ and
$\frac1{w_{n+1}^\gamma}\leq d(x,y)^\gamma$.
\end{proof}

\begin{Remark}\label{lacunary}
Recall that if $(v_n)_{n=1}^\infty$ is a lacunary sequence, i.e.\
there exists $A>1$ such that $v_{n+1}>Av_n$ for all $n\geq 1$,
then for each $1\leq k< n$ we have $v_n>A^{n-k}v_k$, so
\[\sum_{k=1}^nv_k<\sum_{k=1}^n\frac{v_n}{A^{n-k}}<v_n\frac{A}{A-1}\]
and (by changing the role of $n$ and $k$)
\[\sum_{k=n}^\infty \frac{1}{v_k}<\sum_{k=n}^\infty
\frac{1}{A^{k-n}v_n}<\frac{1}{v_n}\frac{A}{A-1}.\] Note that if
$(v_n)_{n=1}^\infty$ is lacunary, then for each $0<\beta<1$ also
$(v_n^\beta)_{n=1}^\infty$ is lacunary (with $A$ replaced by
$A^\beta$) and therefore the assumption~(\ref{asla}) is satisfied
in this case.
\end{Remark}

Let $(q_n)_{n=0}^\infty$ be the sequence of denominators of an
irrational number $\alpha\in\T$. Note that
\[q_{n+2}=a_{n+2}q_{n+1}+q_n\geq q_{n+1}+q_n>2q_n.\]
Therefore, the sequence $(w_n)_{n=0}^\infty$, $w_n:=q_nq_{n+1}$
 is lacunary with $A=2$. It follows that
for each $0<\beta<1$ setting $D_\beta=2^{\beta}/(2^{\beta}-1)$ we
have
\begin{equation}\label{mian}
\sum_{k=1}^nw_k^\beta\leq D_\beta w^\beta_n\;\;\mbox{ and
}\sum_{k=n}^{\infty}\frac{1}{w_k^\beta}\leq\frac{D_\beta}{w_n^{\beta}}\mbox{
for all }n\in\N.
\end{equation}

\begin{Notation}
For every $a\geq 1$ denote by $DC(a)$ the set of irrational
numbers $\alpha\in\T$ satisfying the  Diophantine condition
\[\left|\alpha-\frac{p}{q}\right|\geq\frac{1}{C|q|^{1+a}}\text{ for all }p,q\in \Z,
\;q\neq 0\]
for some constant $C>0$. Recall that (see \cite{Kh}) $\alpha\in
DC(a)$ if and only if there exists $C>0$ such that $q_{n+1}\leq
C\cdot q_n^{a}$ for all $n\geq 1$. Moreover, if $a>1$ then $DC(a)$
has full Lebesgue measure.
\end{Notation}

\begin{Th}
Assume that $\alpha\in DC(a)$ for some $a\geq 1$. Then for every
$0<\gamma<\frac1{(1+a)a}$ there exists a $\gamma$--H\"older
continuous function $\va:\T\to\R$ with zero mean such that the
cylindrical transformation $T_{\va}:\T\times\R\to\T\times\R$,
$T_{\va}(x,s)=(x+\alpha,s+\va(x))$ is Besicovitch.
\end{Th}

\begin{proof}
Let $C$ be a positive constant such that
\begin{equation}\label{diof}
q_{n+1}\leq C\cdot q_n^{a}\text{ for all }n\geq 1.
\end{equation}
For any $0<\gamma<\frac1{(1+a)a}$ we set
$$
M_n= \frac{q_{n-1}}{2(q_nq_{n+1})^\gamma}.$$ Then we construct
$\va$ in the same manner as in Section~\ref{secpodst}. By
(\ref{diof}),
$$
\frac{q_{n-1}}{(q_nq_{n+1})^\gamma}\geq\frac{q_{n-1}}{C^{\gamma}q_{n-1}
^{\gamma a}C^{\gamma(1+a)}q_{n-1}^{\gamma
a^2}}=\frac{q_{n-1}^{1-\gamma(1+a)a}}{C^{\gamma(2+a)}},$$ and
therefore (\ref{m2}) holds. But $M_n/q_{n-1}=1/(2w_n^\gamma)$ and
$(w_n^\gamma)_{n=0}^\infty$ is lacunary, so
 also (\ref{m3}) is satisfied. Let
$\varphi_n(x)=f_n(x+\alpha)-f_n(x)$. Then
$$L(\va_n)=2L_n=2M_n\cdot\frac{q_nq_{n+1}}{q_{n-1}}=(q_nq_{n+1})^{1-\gamma}
=w_n^{1-\gamma}$$
and, by (\ref{normsup}) and (\ref{m4}),
$$
\|\va_n\|_{C(\T)}\leq\frac{L_n}{q_nq_{n+1}}=\frac{M_n}{q_{n-1}}\leq
\frac1{w_n^\gamma}.$$
 By Lemma~\ref{kf1} together with
(\ref{mian}), the function $\va$ is $\gamma$--H\"older continuous.
\end{proof}

\section{Function of class $\mbox{O}(\log |n|/|n|)$}\label{oduze}
Assume that $\alpha\in[0,1)$ is irrational and suppose that
$(M_m)_{m\geq 1}$ satisfies (\ref{m2}) and (\ref{m3}). Set
$\delta_m=\frac{q_{m-1}}{q_mq_{m+1}}$. Then $M_m=L_m\delta_m$. Let
us consider a piecewise linear function $f_m:[0,1)\to\R^+$ which
is $1/q_{m}$--periodic and
\[
f_m(x)=\left\{
\begin{array}{ccc}
L_mx& \text{if} & 0\leq x\leq \delta_m\\
M_m& \text{if}& \delta_m\leq x\leq\frac{1}{q_m}-\delta_m\\
L_m(\frac{1}{q_m}-x)& \text{if} &\frac{1}{q_m}-\delta_m\leq x\leq
\frac{1}{q_m}.
\end{array}\right.\]
Next, consider $\va(x)=\sum_{m=1}^\infty(f_m(x+\alpha)-f_m(x))$.
An analysis similar to that in Section~\ref{secpodst} shows that
$\va$ is correctly defined and it is continuous. Moreover,
$\va^{(k)}(0)\to+\infty$ as $|k|\to+\infty$. Let
$g_m:[0,1)\to\R^+$ be given by
\[g_m(x)=\left\{
\begin{array}{ccc}
x& \text{if} & 0\leq x\leq q_m\delta_m\\
q_m\delta_m& \text{if}& q_m\delta_m\leq x\leq 1-q_m\delta_m\\
1-x& \text{if} &1-q_m\delta_m\leq x\leq 1.
\end{array}\right.\] Therefore $f_m(x)=\frac{L_m}{q_m}g_m(q_mx)$
for $x\in[0,1/q_m)$ and for $n\neq 0$ we have
\begin{eqnarray*}\widehat{g}_m(n)&=&\int_0^1g_m(x)e^{-2\pi inx}\,dx=\frac1{2\pi
in}\int_0^1g_m'(x)e^{-2\pi inx}\,dx\\&=& \frac1{2\pi
in}\left(\int_0^{q_m\delta_m}e^{-2\pi inx}\,dx -
\int_{1-q_m\delta_m}^1e^{-2\pi inx}\,dx \right)\\&=&
\frac1{4\pi^2n^2}\left(e^{-2\pi inq_m\delta_m}+e^{2\pi
inq_m\delta_m}-2\right)=-\frac{1}{\pi^2n^2}\sin^2\pi nq_m\delta_m.
\end{eqnarray*}
Since $f_k$ is $1/q_k$--periodic,
\[\widehat{f}_k(n)=\sum_{j=0}^{q_k-1}e^{-2\pi in\frac{j}{q_k}}
\int_0^{1/q_k}f_k(x)e^{-2\pi inx}\,dx,\] whence
$\widehat{f}_k(n)=0$ if $q_k$ does not divide $n$ and moreover
\begin{eqnarray*}
\widehat{f}_k(q_ks)&=&q_k\int_0^{1/q_k}f_k(x)e^{-2\pi iq_ksx}\,dx=
q_k\int_0^{1/q_k}\frac{L_k}{q_k}g_k(q_kx)e^{-2\pi iq_kx}\,dx\\&=&
L_k\int_0^1g_k(y)e^{-2\pi
isy}\left(\frac1{q_k}\,dy\right)=\frac{L_k}{q_k}\widehat{g}_k(s)
\end{eqnarray*}
for each $s\in\Z$. It follows that
\begin{eqnarray*}
\widehat{\va}(n)&=&(e^{2\pi
in\al}-1)\sum_{k=1}^\infty\widehat{f}_k(n)= (e^{2\pi
in\al}-1)\sum_{k\geq 1
:q_k|n}\frac{L_k}{q_k}\widehat{g}_k(\frac{n}{q_k})\\&=& (e^{2\pi
in\al}-1)\sum_{k\geq 1:q_k|n}\frac{L_k}{q_k}\cdot\frac{-\sin^2\pi
n\delta_k}{\pi^2\frac{n^2}{q_k^2}}\\&=&\frac{1-e^{2\pi
in\al}}{\pi^2n^2}\sum_{k\geq 1:q_k|n}L_kq_k\sin^2\pi
n\delta_k.\end{eqnarray*} Thus
\begin{equation}\label{furie}
|\widehat{\va}(n)|=|\widehat{\va}(-n)|=\frac{2|\sin\pi
n\al|}{\pi^2n^2}\sum_{k\geq 1:q_k|n}L_kq_k\sin^2\pi
n\delta_k.
\end{equation}

\begin{Remark}Recall that
\begin{equation}\label{singora}
|\sin \pi x|\leq \pi\|x\|\leq \pi|x| \;\;\text{ for all
}\;\;x\in\R
\end{equation}
and
\begin{equation}\label{sindul}
\sin \pi x\geq 2x \;\;\text{ for all }\;\;x\in[0,1/2].
\end{equation}\end{Remark}

\begin{Lemma} Assume that $\alpha\in[0,1)$ is an irrational number
such that
\[\left(\frac{\log q_k}{kq^2_{k-1}}\right)_{k\geq 1}\text{ is increasing and }
\frac{\log q_k}{kq^2_{k-1}}\to+\infty\text{ as }k\to+\infty.\]Then
there exists  $(M_k)_{k\in\N}$ with $M_k\to+\infty$ such that for
$\va$ above we have $\widehat{\varphi}(n)=\mbox{\em O}(\log
|n|/|n|)$.
\end{Lemma}
\begin{proof}Fix $0<\vep<1$ and let
\[M_k=\min\left(\frac{\log q_k}{kq^2_{k-1}},q_{k-1}^\vep\right).\]
Then $(M_k)_{k\geq 1}$ is increasing and $M_k\to+\infty$. Next
note that the sequence $(\delta_k)_{k=1}^{\infty}$ is decreasing.
Indeed, since $q^2_{k}<q^2_{k+1}<\log q_{k+2}<q_{k+2}$, it follows
that
\[\delta_{k}=\frac{q_{k-1}}{q_{k}q_{k+1}}>\frac{q_{k-1}q_k}{q_{k+1}q_{k+2}}\geq
\frac{q_k}{q_{k+1}q_{k+2}}=\delta_{k+1}.\]
Since $L_k=M_k/\delta_k$, we see that $(L_k)_{k=1}^{\infty}$ is
increasing. Fix $n>0$ and let $m\geq 0$ be the largest number such
that $q_m$ divides $n$. Note that if $m=0$, i.e.\ $q_k$ does not
divide $n$ for all $k\geq 1$, then $\widehat{\va}(n)=0$, so assume
that $m\geq 1$.

First suppose that $n\geq q_{m+1}$. Since
$mq^2_{m}M_{m}\leq(m+1)q^2_{m}M_{m+1}\leq \log q_{m+1}$, by
(\ref{furie}),
\begin{eqnarray*}|\widehat{\va}(n)|&\leq&\frac{2}{\pi^2n^2}
\sum_{k\geq 1:q_k|n}L_kq_k\leq
\frac{2mL_mq_m}{\pi^2n^2}\\&\leq&
\frac{2mL_mq_m}{q_{m+1}}\frac{1}{n}=
\frac{2mM_mq^2_m}{q_{m-1}}\frac{1}{n}\leq \frac{\log
q_{m+1}}{n}\leq\frac{\log n}{n}.
\end{eqnarray*}
Next, suppose that $n\leq q_{m+1}$. Since $n=sq_{m}$, by
(\ref{furie}), (\ref{singora}) and (\ref{mip}), we have
\begin{eqnarray*}
|\widehat{\va}(n)|&\leq&\frac{2\|sq_m\al\|}{\pi^2n^2}\sum_{k\geq 1
:q_k|n}L_kq_k\sin^2\pi
n\delta_k\\&\leq&\frac{2s}{\pi^2q_{m+1}n^2}\left(L_mq_m(\pi
n\delta_m)^2+\sum_{k=1}^{m-1}L_kq_k\right).
\end{eqnarray*} Moreover,
\[
\frac{sL_mq_m(\pi n\delta_m)^2}{\pi^2q_{m+1}n^2}=
\frac{sq_m}{q_{m+1}}L_m\delta_{m}^2
=\frac{n}{q_{m+1}}M_m\frac{q_{m-1}}{q_mq_{m+1}}\leq
\frac{M_mq_{m-1}}{q_mq_{m+1}}\leq \frac{1}{q_{m+1}}\leq\frac{1}{n}
\]
and
\begin{eqnarray*} \frac{s}{\pi^2q_{m+1}n^2}\sum_{k=1}^{m-1}L_kq_k&\leq&
\frac{mL_{m-1}q_{m-1}s}{q_{m+1}n^2}=\frac{mM_{m-1}q^2_{m-1}q_ms}{q_{m-2}q_{m+1}n^2}
\\&=& \frac{mM_{m-1}q^2_{m-1}}{q_{m-2}q_{m+1}n}
\leq \frac{q_{m}}{q_{m-2}q_{m+1}n}\leq\frac{1}{n}.
\end{eqnarray*}
Consequently, $\widehat{\va}(n)=\mbox{O}(\log |n|/|n|)$.
\end{proof}
In this way we have proved the following.
\begin{Th}
There exist an irrational rotation $Tx=x+\alpha$ and a continuous
function $\varphi:\T\to\R$  such that the cylindrical
transformation $T_\varphi$ is Besicovitch  and
$\widehat{\varphi}(n)= \mbox{\em O}(\log |n|/|n|)$.
\end{Th}

\begin{Remark}
Note that this is not true that
$\widehat{\varphi}(n)=\mbox{O}(1/|n|)$. Indeed, take $m\geq 1$ and
choose $s\in\N$ such that
$q_{m+1}/(4q_{m-1})<s<q_{m+1}/(2q_{m-1})$. Set  $n=sq_m$.  Since
\[\|n\alpha\|=\|sq_m\alpha\|\leq s\|q_m\alpha\|<1/2\text{ and }
n\delta_m= s\frac{q_{m-1}}{q_{m+1}}<1/2,\] by (\ref{sindul}) and
(\ref{mip}), it follows that
\[|\sin\pi n\al|\geq
2\|sq_m\alpha\|=2s\|q_m\alpha\|>\frac{1}{4q_{m-1}},\]\[ |\sin\pi
n\delta_m|\geq 2n\delta_m=2s\frac{q_{m-1}}{q_{m+1}}>1/2.\]
Therefore, from (\ref{furie}),
\begin{eqnarray*}| \widehat{\va}(n)|&\geq& \frac{2|\sin\pi
n\al|}{\pi^2n^2}L_mq_m\sin^2\pi n\delta_m\geq
\frac{1}{4\pi^2n^2}\frac{L_mq_m}{q_{m-1}}=\frac{1}{4\pi^2n^2}
\frac{M_mq_{m+1}q^2_m}{q^2_{m-1}}
\\&=&\frac{1}{4\pi^2n}\frac{M_mq_{m+1}q_m}{sq^2_{m-1}}>
\frac{1}{2\pi^2n}\frac{M_mq_m}{q_{m-1}}>\frac{M_m}{2\pi^2n}.
\end{eqnarray*}
Consequently,
\[|n \widehat{\va}(n)|\geq \frac{M_m}{2\pi^2}\to+\infty.\]
\end{Remark}

\section{Variants of the construction}\label{varia}
Fix an irrational $\al\in[0,1)$ and let $(q_{k_n})$ be a
subsequence of the sequence of denominators of $\al$ such that all
$k_n$ are even (or all are odd). Let $(M_n)$ be a sequence of
natural numbers such that
\begin{equation}\label{m2v}M_n\to+\infty\end{equation} and
\begin{equation}\label{m3v}\sum_{n=1}^\infty\frac{M_n}{q_{k_{n-1}}}
<+\infty.\end{equation}
Set
\begin{equation}\label{m4v}L_n=M_n\cdot\frac{q_{k_n}q_{k_n+1}}{q_{k_{n-1}}}\text{
and  }\delta_n=\frac{q_{k_{n-1}}}{q_{k_n}q_{k_n+1}}.\end{equation}
In view of~(\ref{m3v}) and~(\ref{m4v}),
\begin{equation}\label{m5v}\sum_{n=1}^\infty
\frac{L_n}{q_{k_n}q_{k_n+1}}<+\infty.\end{equation}

We now define $f_n:[0,1)\to\R^+$ a piecewise linear Lipschitz
continuous function which $1/q_{k_n}$--periodic and
\[
f_n(x)=\left\{
\begin{array}{ccc}
0& \text{if} & 0\leq x\leq \frac{\delta_n}{8}\\
L_n\left(x-\frac{\delta_n}{8}\right)& \text{if} &
\frac{\delta_n}{8}\leq x\leq \frac{\delta_n}{2}+\frac{\delta_n}{8}\\
\frac{M_n}{2}& \text{if} &
\frac{\delta_n}{2}+\frac{\delta_n}{8}\leq
x\leq \frac{1}{4q_{k_n}}-\frac{\delta_n}{2}\\
L_n\left(x-\frac{1}{4q_{k_n}}\right)+{M_n}& \text{if} &
\frac{1}{4q_{k_n}}-\frac{\delta_n}{2}\leq x\leq
\frac{1}{4q_{k_n}}+\frac{\delta_n}{2}\\
\frac{3M_n}{2}& \text{if} &
\frac{1}{4q_{k_n}}+\frac{\delta_n}{2}\leq
x\leq \frac{1}{2q_{k_n}}-\frac{\delta_n}{2}-\frac{\delta_n}{8}\\
L_n\left(x-\frac{1}{2q_{k_n}}+\frac{\delta_n}{8}\right)+2M_n&
\text{if} &
\frac{1}{2q_{k_n}}-\frac{\delta_n}{2}-\frac{\delta_n}{8}\leq
x\leq \frac{1}{2q_{k_n}}-\frac{\delta_n}{8}\\
2M_n& \text{if} & \frac{1}{2q_{k_n}}-\frac{\delta_n}{8}\leq x\leq
\frac{1}{2q_{k_n}},
\end{array}\right.\]
and such that  $f_n\left({1}/{q_{k_n}}-x\right)=f_n(x)$ for all
$x\in[1/2q_{k_n},1/q_{k_n}]$.

\begin{figure}[h]
\begin{center}
\vspace{7cm}
 \includegraphics{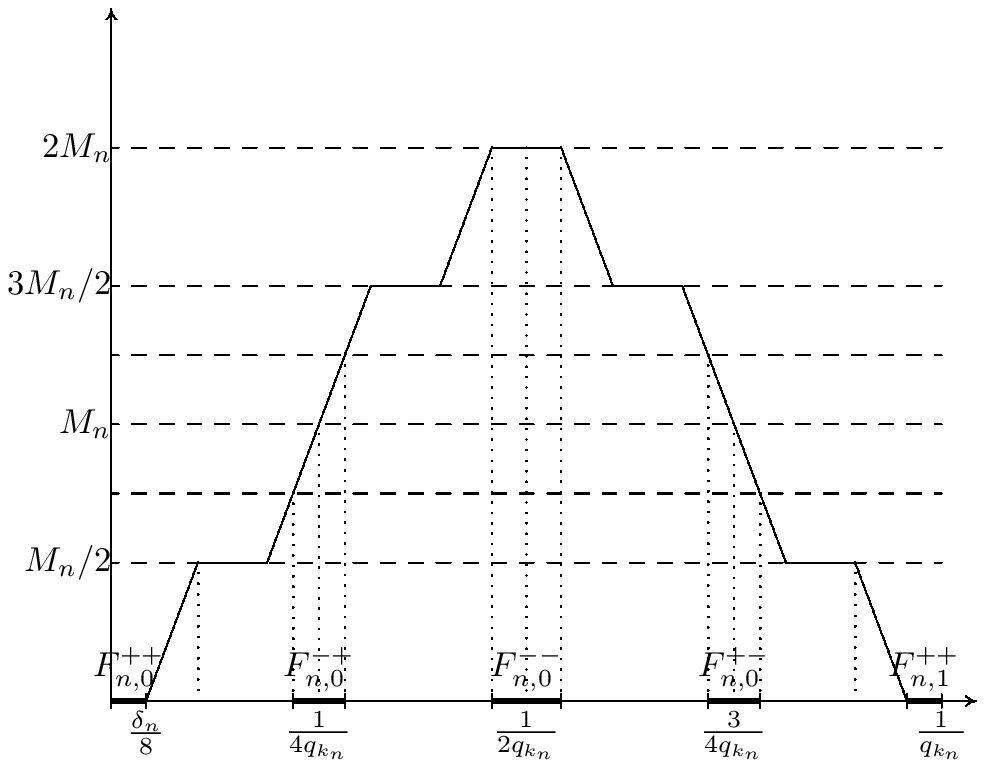} \caption{The graph of $f_n$}
\end{center}
\end{figure}

The Lipschitz constant of $f_n$ is $L_n$, $f_n$ takes only
nonnegative values, $f_n(0)=0$ and $\|f_n\|_{C(\T)}\leq 2M_n$.
Moreover $f_n(-x)=f_n(x)$ for all $x\in[0,1)$. Using
$1/q_{k_n}$--periodicity of $f_n$ and~(\ref{m1}) for each
$x\in[0,1)$ we have
\begin{eqnarray*}
|f_n(x+\al)-f_n(x)|&=&\left|f_n(x+\al)-f_n\left(
x+\frac{p_{k_n}}{q_{k_n}}\right)\right|\\ &\leq&
L_n\left|\al-\frac{p_{k_n}}{q_{k_n}}\right|<L_n\cdot\frac1{q_{k_n}q_{k_n+1}},
\end{eqnarray*} so
$\|f_n(\,\cdot\,+\alpha)-f_n(\,\cdot\,)\|_{C(\T)}<\frac{L_n}{q_{k_n}q_{k_n+1}}$
and it follows from~(\ref{m5v}) that the series
$$
\va(x)=\sum_{n=1}^\infty(f_n(x+\alpha)-f_n(x))$$ converges
uniformly, so $\va$ is continuous and clearly
$\int_0^1\va(x)\,dx=0$. For each integer $k$ we  have
\begin{equation}\label{wzkocykl}
\va^{(k)}(x)=\sum_{n=1}^\infty (f_n(x+k\al)-f_n(x)).
\end{equation}

\begin{Remark} Suppose additionally that
\begin{equation}\label{jednakzal+}
q_{k_{n+1}}>16q_{k_n}q_{k_n+1}/q_{k_{n-1}}\text{ and
}q_{k_n-1}\geq 4q_{k_{n-1}}\text{ for all }n\in\N.
\end{equation}
It follows that
\begin{equation}\label{jednakzal++}
\frac{1}{q_{k_{n}}}-\frac{1}{q_{k_{n}+1}}=
\frac{q_{k_{n}+1}-q_{k_{n}}}{q_{k_{n}}q_{k_{n}+1}}\geq
\frac{q_{k_{n}-1}}{q_{k_{n}}q_{k_{n}+1}}
\geq\frac{4q_{k_{n-1}}}{4q_{k_{n}}q_{k_{n}+1}}=4\delta_n.\end{equation}

Set
\[F^{++}_{n,j}=\left[-\frac{\delta_n}{8},\frac{\delta_n}{8}\right]+\frac{j}{q_{k_n}}\]
and let $F^{++}=\bigcap_{n=1}^\infty
\bigcup_{j=0}^{q_{k_n}-1}F^{++}_{n,j}$. In view of
(\ref{jednakzal+}),
\[|F^{++}_{n,j}|=\frac{\delta_n}{4}=
\frac{q_{k_{n-1}}}{4q_{k_n}q_{k_n+1}}>\frac{4}{q_{k_{n+1}}},\] and
hence there exist at least two intervals of the form
$F^{++}_{n+1,l}$ which are included in $F^{++}_{n,j}$.
Consequently the set $F^{++}$ is uncountable. We will show that if
$x\in F^{++}$ then $\varphi^{(n)}(x)\to+\infty$ as
$|n|\to+\infty$.

Assume that $x\in F^{++}$. Then there exists a sequence
$(j_l)_{l=1}^{\infty}$ of natural numbers such that $x\in
F^{++}_{l,j_l}$ for all $l\in\N$. Let $x=\frac{j_l}{q_{k_l}}+x_l$,
where $|x_l|\leq\delta_l/8$. Then
\[f_l\left(x+m\alpha\right)-f_l\left(x\right)=f_l
\left(x_l+m\alpha\right)-f_l\left(x_l\right).\]

Fix integer $m\neq 0$ and assume that $q_{k_{n-1}}\leq
|m|<q_{k_n}$. Since $|x_l|\leq\delta_l/8$  and $f_l\geq 0$ for
every $l\in\N$, by the definition of $f_l$, $f_l(x_l)=0$, so
\begin{equation}\label{dwapt}
f_l\left(x+m\alpha\right)-f_l\left(x\right)=f_l\left(x_l+m\alpha\right)-
f_l\left(x_l\right)=f_l\left(x_l+m\alpha\right)\geq
0.\end{equation} Since $|m|<q_{k_n}$, in view of (\ref{m1}) and
(\ref{jednakzal++}), we have
\begin{eqnarray*}\left|x_n+m\al-\frac{mp_{k_n}}{q_{k_n}}\right|&\leq&
|x_n|+|m|\left|\al-\frac{p_{k_n}}{q_{k_n}}\right|\leq|x_n|+\frac{|m|}
{q_{k_n}q_{k_{n}+1}}\\
&\leq&\frac{\delta_n}{8}+\frac{1}{q_{k_{n}+1}}<\frac{1}{q_{k_{n}}}-2{\delta_n}.
\end{eqnarray*}
Since $|m|\geq q_{k_{n-1}}$, in view of (\ref{m1}), we have
\[\left|x_n+m\al-\frac{mp_{k_n}}{q_{k_n}}\right|\geq
|m|\left|\al-\frac{p_{k_n}}{q_{k_n}}\right|-|x_n|\geq
\frac{q_{k_{n-1}}}{2q_{k_n}q_{k_n+1}}-\frac{\delta_n}{8}=\frac{\delta_n}{2}
-\frac{\delta_n}{8}=\frac{\delta_n}{8}+\frac{\delta_n}{4}.\] By
the definition of $f_n$ it follows that
\[f_n\left(\left|x_n+m\al-\frac{mp_{k_n}}{q_{k_n}}\right|\right)
\geq f_n\left(\frac{\delta_n}{8}+\frac{\delta_n}{4}\right)=
L_n\frac{\delta_n}{4}=M_n/4.\]
Therefore, using additionally (\ref{dwapt}), we obtain
\begin{eqnarray*}f_n\left(x+m\alpha\right)-f_n\left(x\right)
&=&f_n\left(x_n+m\alpha\right)=
f_n\left(x_n+m\al-\frac{mp_{k_n}}{q_{k_n}}\right)\\
&=&f_n\left(\left|x_n+m\al-\frac{mp_{k_n}}{q_{k_n}}\right|\right)\geq
M_n/4.
\end{eqnarray*}
Consequently, using (\ref{wzkocykl}) and (\ref{dwapt}) again,
\[\va^{(m)}(x)=\sum_{l=1}^\infty(f_l\left(x+m\alpha\right)-
f_l\left(x\right))\geq f_n\left(x+m\alpha\right)-f_n\left(x\right)\geq\frac{M_n}{4}.\]
\end{Remark}
\begin{Remark} Suppose additionally that
\begin{equation}\label{jednakzal} M_{n+1}\geq 33M_n\text{ for all }n\in\N.
\end{equation}
We will prove that the set of all $x\in[0,1)$ for which
$\varphi^{(m)}(x)\to+\infty$ as $m\to+\infty$ and simultaneously
$\varphi^{(m)}(x)\to-\infty$ as $m\to-\infty$ is uncountable.

Set
\[F^{-+}_{n,j}=\left[-\frac{\delta_n}{4},\frac{\delta_n}{4}\right]+
\frac{1}{4q_{k_n}}+\frac{j}{q_{k_n}}\]
and let $F^{-+}=\bigcap_{n=1}^\infty
\bigcup_{j=0}^{q_{k_n}-1}F^{-+}_{n,j}$. In view of
(\ref{jednakzal+}),
\[|F^{-+}_{n,j}|=\frac{\delta_n}{2}=\frac{q_{k_{n-1}}}{2q_{k_n}q_{k_n+1}}>
\frac{4}{q_{k_{n+1}}},\]
and hence there exist at least two intervals of the form
$F^{-+}_{n+1,l}$ which are included in $F^{-+}_{n,j}$.
Consequently the set $F^{-+}$ is uncountable. We will prove that
if $x\in F^{-+}$ then $\varphi^{(m)}(x)\to+\infty$ as
$m\to+\infty$ and $\varphi^{(m)}(x)\to-\infty$ as $m\to-\infty$.

Assume that $x\in F^{-+}$. Then there exists a sequence
$(j_l)_{l=1}^{\infty}$ for natural numbers such that $x\in
F^{-+}_{l,j_l}$ for all $l\in\N$. Let
$x=\frac{j_l}{q_{k_l}}+\frac{1}{4q_{k_l}}+x_l$, where
$|x_l|\leq\delta_l/4$. Then
\begin{equation}\label{trzypt}
f_l\left(x+m\alpha\right)-f_l\left(x\right)=f_l\left(\frac{1}
{4q_{k_l}}+x_l+m\alpha\right)-
f_l\left(\frac{1}{4q_{k_l}}+x_l\right).
\end{equation}
For every $l\geq 1$ we have
\begin{equation*}
\frac{1}{4q_{k_l}}-\frac{\delta_l}{4}<\frac{1}{4q_{k_l}}+x_l<
\frac{1}{4q_{k_l}}+\frac{\delta_l}{4},
\end{equation*}
and hence
\begin{equation}\label{locflx}
f_l\left(\frac{1}{4q_{k_l}}+x_l\right)=L_lx_l+M_l.
\end{equation}

Fix $m>0$ and assume that $q_{k_{n-1}}/4\leq m<q_{k_n}/4$. Since
$k_l$ is even, for every $l\geq 1$ we have
\[\frac{1}{4q_{k_l}}-\frac{\delta_l}{4}<\frac{1}{4q_{k_l}}+x_l\leq
\frac{1}{4q_{k_l}}+x_l+m\al-\frac{mp_{k_l}}{q_{k_l}}.\]
If additionally $l\geq n$ then, by (\ref{m1}),
\[\frac{1}{4q_{k_l}}+x_l+m\al-\frac{mp_{k_l}}{q_{k_l}}<
\frac{1}{4q_{k_l}}+\frac{\delta_l}{4}+\frac{q_{k_n}}{4q_{k_l}q_{k_{l}+1}}<
\frac{1}{4q_{k_l}}+\frac{\delta_l}{4}+\frac{1}{4q_{k_{l}+1}}.\] In
view of (\ref{jednakzal++}), $\frac{1}{q_{k_l+1}}\leq
\frac{1}{q_{k_l}}-4\delta_l$.
 It
follows that
\[\frac{1}{4q_{k_l}}+x_l+m\al-\frac{mp_{k_l}}{q_{k_l}}<
\frac{1}{2q_{k_l}}-\frac{3\delta_l}{4},\]
hence
\[\frac{1}{4q_{k_l}}+x_l+m\al-\frac{mp_{k_l}}{q_{k_l}}\in\left({1}/
{(4q_{k_l})}-{\delta_l}/{4},{1}/{(2q_{k_l})}-{3\delta_l}/{4}\right).\]

{\em Case 1.} Suppose that $l\geq n$ and
$$\frac{1}{4q_{k_l}}+x_l+m\al-\frac{mp_{k_l}}{q_{k_l}}\in\left({1}/{(4q_{k_l})}-
{\delta_l}/{4},{1}/{(4q_{k_l})}+{\delta_l}/{2}\right].$$ By the
definition of $f_l$,
\begin{eqnarray*}f_l\left(\frac{1}{4q_{k_l}}+x_l+m\alpha\right)&=&
f_l\left(\frac{1}{4q_{k_l}}+x_l+m\al-\frac{mp_{k_l}}{q_{k_l}}\right)
\\&=&L_l\left(x_l+m\al-\frac{mp_{k_l}}{q_{k_l}}\right)+M_l.
\end{eqnarray*}
Hence, by (\ref{trzypt}), (\ref{locflx}) and (\ref{m1}),
\begin{eqnarray*}\lefteqn{f_l\left(x+m\alpha\right)-f_l\left(x\right)
=f_l\left(\frac{1}{4q_{k_l}}+x_l+m\alpha\right)-f_l\left(\frac{1}{4q_{k_l}}+
x_l\right)}\\
&=&L_l\left(m\al-\frac{mp_{k_l}}{q_{k_l}}\right) \geq
L_l\frac{q_{k_{n-1}}}{8q_{k_l}q_{k_l+1}}=M_l\frac{q_{k_{n-1}}}
{8q_{k_{l-1}}}.\;\;\;\;\;\;\;\;\;\;\;\;\;\;\;\;\;\;\;\;\;
\end{eqnarray*}

{\em Case 2.} Suppose that $l\geq n$ and
$$\frac{1}{4q_{k_l}}+x_l+m\al-\frac{mp_{k_l}}{q_{k_l}}\in
\left({1}/{(4q_{k_l})}+{\delta_l}/{2},{1}/{(2q_{k_l})}-{3\delta_l}/{4}\right).$$
By the definition of $f_l$,
\[f_l\left(\frac{1}{4q_{k_l}}+x_l+m\alpha\right)
=f_l\left(\frac{1}{4q_{k_l}}+x_l+m\al-\frac{mp_{k_l}}{q_{k_l}}\right)=\frac{3}{2}M_l.\]
Moreover, by (\ref{locflx}),
\[f_l\left(\frac{1}{4q_{k_l}}+x_l\right)=L_lx_l+M_l\leq L_l
\frac{\delta_l}{4}+M_l=\frac{5}{4}M_l.\]
Therefore, using (\ref{trzypt}),
\[f_l\left(x+m\alpha\right)-f_l\left(x\right)
=f_l\left(\frac{1}{4q_{k_l}}+x_l+m\alpha\right)-f_l\left(\frac{1}
{4q_{k_l}}+x_l\right)\geq\frac{1}{4}M_l.\]

In summary, we have
$f_n\left(x+m\alpha\right)-f_n\left(x\right)\geq M_n/8$ and
$f_l\left(x+m\alpha\right)-f_l\left(x\right)>0$ for $l>n$.
Moreover, $|f_l\left(x+m\alpha\right)-f_l\left(x\right)|\leq 2M_l$
for all $l\geq 1$, in particular for $l<n$. By (\ref{wzkocykl})
and (\ref{jednakzal}), it follows that
\[\va^{(m)}(x)=\sum_{l=1}^\infty(f_l\left(x+m\alpha\right)-f_l\left(x\right))
\geq
\frac{M_n}{8}-\sum_{l=1}^{n-1}2M_l\geq\frac{M_n}{8}-2\frac{M_n}{32}\geq\frac{M_n}{16}\]
whenever  $q_{k_{n-1}}/4\leq m<q_{k_n}/4$. Consequently,
$\va^{(m)}(x)\to+\infty$ as $m\to+\infty$.

Similar argument will show also that $\va^{(-m)}(x)\leq -M_n/16$
whenever $x\in F^{-+}$ and $q_{k_{n-1}}/4\leq -m<q_{k_n}/4$. Since
$k_l$ is even, $|x_l|\leq\delta_l/8$ and $m$ is negative, for
every $l\geq 1$ we have
\[\frac{1}{4q_{k_l}}+x_l+m\left(\al-\frac{p_{k_l}}{q_{k_l}}\right)\leq
\frac{1}{4q_{k_l}}+x_l<\frac{1}{4q_{k_l}}+\frac{\delta_l}{4}.\]
Suppose additionally that $l\geq n$. Since $m>-q_{k_n}/4$, by
(\ref{m1}) and (\ref{jednakzal++}),
\begin{eqnarray*}\frac{1}{4q_{k_l}}+x_l+m\left(\al-\frac{p_{k_l}}{q_{k_l}}\right)&>&
\frac{1}{4q_{k_l}}-\frac{\delta_l}{4}+\frac{m}{q_{k_l}q_{k_{l}+1}}>
\frac{1}{4q_{k_l}}-\frac{\delta_l}{4}-\frac{q_{k_n}}{4q_{k_l}q_{k_{l}+1}}\\&>&
\frac{1}{4q_{k_l}}-\frac{1}{4q_{k_{l}+1}}-\frac{\delta_l}{4}>\frac{3\delta_l}{4}.
\end{eqnarray*}
Hence
\[\frac{1}{4q_{k_l}}+x_l+m\al-\frac{mp_{k_l}}{q_{k_l}}\in
\left({3\delta_l}/{4},{1}/{(4q_{k_l})}+{\delta_l}/{4}\right).\]

{\em Case 1.} Suppose that $l\geq n$ and
$$\frac{1}{4q_{k_l}}+x_l+m\al-\frac{mp_{k_l}}
{q_{k_l}}\in\left({1}/{(4q_{k_l})}-{\delta_l}/{2},{1}/{(4q_{k_l})}+
{\delta_l}/{4}\right].$$
By the definition of $f_l$,
\begin{eqnarray*}f_l\left(\frac{1}{4q_{k_l}}+x_l+m\alpha\right)&=&
f_l\left(\frac{1}{4q_{k_l}}+x_l+m\al-\frac{mp_{k_l}}{q_{k_l}}\right)
\\&=&L_l\left(x_l+m\al-\frac{mp_{k_l}}{q_{k_l}}\right)+M_l.
\end{eqnarray*}
Hence, by (\ref{trzypt}), (\ref{locflx}) and (\ref{m1}),
\begin{eqnarray*}\lefteqn{f_l\left(x+m\alpha\right)-f_l\left(x\right)
=f_l\left(\frac{1}{4q_{k_l}}+x_l+m\alpha\right)-f_l
\left(\frac{1}{4q_{k_l}}+x_l\right)}\\
&=&L_lm\left(\al-\frac{p_{k_l}}{q_{k_l}}\right) \leq
-L_l\frac{q_{k_{n-1}}}{8q_{k_l}q_{k_l+1}}=-M_l
\frac{q_{k_{n-1}}}{8q_{k_{l-1}}}.\;\;\;\;\;\;\;\;\;\;\;\;\;\;\;\;\;\;\;\;\;
\end{eqnarray*}

{\em Case 2.} Suppose that $l\geq n$ and
$$\frac{1}{4q_{k_l}}+x_l+m\al-\frac{mp_{k_l}}
{q_{k_l}}\in\left({3\delta_l}/{4},{1}/{(4q_{k_l})}-{\delta_l}/{2}\right).$$
By the definition of $f_l$,
\[f_l\left(\frac{1}{4q_{k_l}}+x_l+m\alpha\right)
=f_l\left(\frac{1}{4q_{k_l}}+x_l+m\al-\frac{mp_{k_l}}{q_{k_l}}\right)=M_l/2.\]
Moreover, by (\ref{locflx}),
\[f_l\left(\frac{1}{4q_{k_l}}+x_l\right)=L_lx_l+M_l\geq -
L_l\frac{\delta_l}{4}+M_l=\frac{3}{4}M_l.\]
Therefore, using (\ref{trzypt}),
\[f_l\left(x+m\alpha\right)-f_l\left(x\right)
=f_l\left(\frac{1}{4q_{k_l}}+x_l+m\alpha\right)-
f_l\left(\frac{1}{4q_{k_l}}+x_l\right)\leq-\frac{1}{4}M_l.\]

In summary, we have
$f_n\left(x+m\alpha\right)-f_n\left(x\right)\leq -M_n/8$ and
$f_l\left(x+m\alpha\right)-f_l\left(x\right)<0$ for $l>n$.
Moreover, $f_l\left(x+m\alpha\right)-f_l\left(x\right)\leq 2M_l$
for all $l\geq 1$. By (\ref{wzkocykl}) and (\ref{jednakzal}), it
follows that
\[\va^{(m)}(x)=\sum_{l=1}^\infty(f_l\left(x+m\alpha\right)-f_l\left(x\right))
\leq
-\frac{M_n}{8}+\sum_{l=1}^{n-1}2M_l\leq-\frac{M_n}{8}+
2\frac{M_n}{32}\leq-\frac{M_n}{16}\]
whenever  $q_{k_{n-1}}/4\leq -m<q_{k_n}/4$. Consequently,
$\va^{(m)}(x)\to-\infty$ as $m\to-\infty$.
\end{Remark}
\begin{Remark}
In a similar way we can prove that the sets
\[F^{--}=\bigcap_{n=1}^\infty \bigcup_{j=0}^{q_{k_n}-1}
\left(\left[-\frac{\delta_n}{8},\frac{\delta_n}{8}\right]+
\frac{1}{2q_{k_n}}+\frac{j}{q_{k_n}}\right)\] and
\[F^{+-}=\bigcap_{n=1}^\infty
\bigcup_{j=0}^{q_{k_n}-1}\left(\left[-\frac{\delta_n}{4},
\frac{\delta_n}{4}\right]+\frac{3}{4q_{k_n}}+\frac{j}{q_{k_n}}\right)\]
are  uncountable and  $x\in F^{--}$ implies
$\va^{(m)}(x)\to-\infty$ as $|m|\to+\infty$ and  $x\in F^{+-}$
implies $\va^{(m)}(x)\to+\infty$ as $m\to-\infty$ and
$\va^{(m)}(x)\to-\infty$ as $m\to+\infty$.
\end{Remark}

\begin{Prop}
For every irrational $\alpha\in\T$ there exists a continuous
function $\varphi:\T\to\R$ with zero mean such that the set
$D^{s_-s_+}(\alpha,\varphi)$ is uncountable for every
$s_-,s_+\in\{-,+\}$.\bez
\end{Prop}

In the next section, under some additional assumptions on
$\alpha$, we will show that the sets $D^{s_-s_+}(\alpha,\varphi)$
may have positive Hausdorff dimension.

\section{Hausdorff dimension of $D^{s_-s_+}(\alpha,\varphi)$}
\label{secdimension} Let $(E_k)_{k\geq 0}$ be a sequence of
subsets of $[0,1]$ such that each $E_k$ is a union of a finite
number of disjoint closed intervals (called $k$--th level basic
intervals). Suppose that  each interval of $E_{k-1}$ includes at
least $m_k\geq 2$ intervals of $E_{k}$, and the maximum length of
$k$--th level intervals tends to zero as $k\to+\infty$. Let us
consider the set
\[F=\bigcap_{k=0}^{\infty}E_k.\]
We will use the following criterion to estimate the Hausdorff
dimension from below.
\begin{Prop}[see Example~4.6 in \cite{Fal} and its proof]\label{propfal}
Suppose that  the $k$--th level intervals are separated by gaps of
length at least $\vep_k$ so that $\vep_k\geq\vep_{k+1}>0$ for
every $k\in\N$. Then
\[\dim_HF\geq\liminf_{k\to\infty}\frac{\log (m_1\ldots m_{k-1})}{-\log(m_k\vep_k)}.\]
\end{Prop}

\vspace{1ex}

\begin{Remark}\label{liczint}
Let us consider two intervals $A,B\subset\R$ of length $a$ and $b$
respectively and let $h>0$. Suppose that $a<h<b$. Then there exist
at least $\lfloor\frac{b-a}{h}\rfloor$ intervals of the form
$A+kh$, $k\in\Z$ included in $B$. Moreover, if $b\geq 4h$ then
$\lfloor\frac{b-a}{h}\rfloor\geq \frac{b}{4h}$.
\end{Remark}

\vspace{1ex}

Fix $s_+,s_-\in\{+,-\}$ and let
$$E^{s_-s_+}_n:=\bigcup_{j=0}^{q_{k_n}-1}F^{s_-s_+}_{n,j}=
\bigcup_{j=0}^{q_{k_n}-1}(F^{s_-s_+}_{n,0}+\frac{j}{q_{k_n}}).$$
Then $F^{s_-s_+}=\bigcap_{n=0}^\infty E^{s_-s_+}_n$. As we have
already noticed
\[|F^{s_-s_+}_{n,j}|\geq\delta_n/4\geq\frac{4}{q_{k_{n+1}}}\text{
and}\frac{1}{q_{k_{n+1}}}\geq
\delta_{n+1}\geq|F^{s_-s_+}_{n+1,l}|,\] so $|F^{s_-s_+}_{n,j}|\geq
4h$ and $h\geq |F^{s_-s_+}_{n+1,l}|$, where $h=1/q_{k_n+1}$, and
then by Remark~\ref{liczint},
\[
m_{n+1}\geq\left\lfloor\frac{|F^{s_-s_+}_{n,j}|-|F^{s_-s_+}_{n+1,l}|}{h}\right\rfloor
\geq\frac{|F^{s_-s_+}_{n,j}|}{4h}\geq\frac{\delta_nq_{k_{n+1}}}{16}=
\frac{q_{k_{n-1}}q_{k_{n+1}}}{16q_{k_{n}}q_{k_{n}+1}}.
\]
Moreover,
\[\vep_n=\frac{1}{q_{k_{n}}}-|F^{s_-s_+}_{n,0}|\geq\frac{1}{q_{k_{n}}}-
\frac{\delta_n}{2}\geq\frac{1}{2q_{k_{n}}}.\]
It follows that
\[m_1\ldots m_{n-1}\geq m_{n-1}\geq\frac{q_{k_{n-3}}q_{k_{n-1}}}
{16q_{k_{n-2}}q_{k_{n-2}+1}}
\geq\frac{q_{k_{n-1}}}{16q_{k_{n-2}}q_{k_{n-2}+1}}.\] Moreover,
\[m_n\vep_n\geq\frac{q_{k_{n-2}}q_{k_{n}}}{16q_{k_{n-1}}q_{k_{n-1}+1}}
\frac{1}{2q_{k_{n}}}=\frac{q_{k_{n-2}}}{32q_{k_{n-1}}q_{k_{n-1}+1}}>
\frac{1}{32q_{k_{n-1}}q_{k_{n-1}+1}}.\] Thus
\begin{equation}\label{szaha}
\frac{\log (m_1\ldots m_{n-1})}{-\log(m_n\vep_n)}\geq\frac{\log
q_{k_{n-1}}-\log 16-\log q_{k_{n-2}}-\log q_{k_{n-2}+1}}{ \log
q_{k_{n-1}}+\log q_{k_{n-1}+1}+\log 32}.
\end{equation}

\begin{Th}\label{wymhaus}
Suppose that there exist $\gamma\geq 1$ and $C>0$ such that
\begin{equation}\label{sldiof}
q_{n+1}\leq Cq_{n}^\gamma\text{ for infinitely many }n\in\N.
\end{equation}
Then there exists a continuous function $\va:\T\to\R$ with
zero mean such that $\dim_HD^{s_-s_+}(\alpha,\va)\geq
1/(1+\gamma)$ for all $s_+,s_-\in\{+,-\}$.
\end{Th}

\begin{proof}
By assumption, we can find a subsequence $(q_{k_n})$ of even
denominators of $\alpha$ (or odd) such that
\begin{equation}\label{wzrost}
q_{k_n}\geq(q_{k_{n-1}}q_{k_{n-1}+1})^n, \;q_{k_n-1}\geq
4q_{k_{n-1}}\text{ and }q_{k_{n}+1}\leq Cq^\gamma_{k_n}\text{ for
all }n\in\N.
\end{equation} Take $M_n=33^n$ and let us consider the
function $\varphi$ constructed in Section~\ref{varia}. Then
(\ref{m2v}), (\ref{m3v}), (\ref{jednakzal+}) and (\ref{jednakzal})
hold. It follows that $F^{s_-s_+}\subset D^{s_-s_+}(\alpha,\va)$.
Moreover, by Proposition~\ref{propfal}, (\ref{szaha}) and
(\ref{wzrost}),
\begin{eqnarray*}
\dim_HF^{s_-s_+}&\geq&\liminf_{n\to\infty}\frac{\log (m_1\ldots
m_{n-1})}{-\log(m_n\vep_n)}\\
&\geq&\liminf_{n\to\infty}\frac{\log q_{k_{n-1}}-\log 16-\log
q_{k_{n-2}}-\log q_{k_{n-2}+1}}{ \log q_{k_{n-1}}+\log
q_{k_{n-1}+1}+\log 32}\\
&\geq&\lim_{n\to\infty}\frac{(1-1/n)\log q_{k_{n-1}}-\log 16}{
(1+\gamma)\log q_{k_{n-1}}+\log C+\log 32}=\frac{1}{1+\gamma},\\
\end{eqnarray*}
which completes the proof.
\end{proof}

\begin{Remark}
In particular, for every $\alpha\in DC(\gamma)$  the condition
(\ref{sldiof}) holds for some $C>0$, so the statement of
Theorem~\ref{wymhaus} remains valid.
\end{Remark}

\begin{Remark} It is of course true that
$\bigcup_{n\in\Z}\left(F^{s_-s_+}+n\alpha\right)\subset
D^{s_-s_+}(\alpha,\va)$, however
$$\dim_H\bigcup_{n\in\Z}\left(F^{s_-s_+}+n\alpha\right)=\sup_{n\in\Z}
\dim_H\left(F^{s_-s_+}+n\alpha\right)= \dim_H F^{s_-s_+}.$$ Notice
also that for each $n\geq1$ the set $F^{s_-s_+}$ is covered by
$q_{k_n}$ intervals each of which has length $\delta_n/4$. By
Proposition~4.1 in \cite{Fal}, using (\ref{wzrost}), it follows
that
$$\dim_H F^{s_-s_+}\leq \liminf_{n\to\infty}\frac{\log
q_{k_n}}{-\log
\left(q_{k_{n-1}}/(4q_{k_n}q_{k_n+1})\right)}\leq\frac12.$$ It
follows that if we use only sets $F^{s_-s_+}$ we cannot estimate
the Hausdorff dimension of $D^{s_-s_+}(\alpha,\va)$ from below by
a number greater than $1/2$. Note finally that if $\alpha$ has
bounded partial quotients then $\dim_H F^{s_-s_+}=\frac12$.
\end{Remark}

\begin{Th}\label{wym12}
For almost every $\alpha\in\T$ there exists a continuous function
$\va:\T\to\R$ with zero mean such that
$\dim_HD^{s_-s_+}(\alpha,\va)\geq 1/2$ for all
$s_+,s_-\in\{+,-\}$.
\end{Th}

\begin{proof}
Recall that (see \cite{Kh}) for a.e.\ $\alpha\in\T$ there exist
$C>0$ and an increasing sequence $(k_n)$ of natural numbers such
that
\[q_{k_n+1}\leq C\cdot q_{k_n}\log q_{k_n}\text{ for all }n\in\N.\]
By the proof of Theorem~\ref{wymhaus}, there exists a continuous
function $\va:\T\to\R$ with zero mean such that
\[\dim_HD^{s_-s_+}(\alpha,\va)\geq\liminf_{n\to\infty}
\frac{\log q_{k_{n}}}{ \log q_{k_{n}}+\log
q_{k_{n}+1}}\] for all $s_+,s_-\in\{+,-\}$. However,
\[\frac{\log q_{k_{n}}}{ \log q_{k_{n}}+\log
q_{k_{n}+1}}\geq\frac{\log q_{k_{n}}}{ 2\log q_{k_{n}}+\log\log
q_{k_{n}}+\log C}\to 1/2\] as $n\to+\infty$. Consequently,
$\dim_HD^{s_-s_+}(\alpha,\va)\geq 1/2$.
\end{proof}

\section{Smooth cylindrical transformations  over rotations on higher dimensional
tori}\label{highdim} In this section we will deal with cylindrical
transformations over rotations on higher dimensional tori. More
precisely, we will construct some examples of cylindrical
transformations $T_\va$ of class $C^r$, $r\geq 1$ admitting dense
and discrete orbits and such that the set $D(\alpha,\va)$ has
positive Hausdorff dimension; the construction is based on
Yoccoz's method  in \cite{Yo}.

Fix $d\geq 3$. Let $a>1$ be a real number such that
$\bar{a}:=2a^{d-1}-a^d-1>0$. (The derivative of the function
$P(a)=2a^{d-1}-a^d-1$ is $d-2>0$ at $a=1$ and $P(1)=0$, so $P$
takes positive values on a nonempty interval $(1,b)$.) Let
$\alpha$ be an irrational number such that there exists $C_0>4$
for which
\begin{equation}\label{wieloqn}
4 q_n^{a^d}\leq q_{n+1}\leq C_0 q_n^{a^d}\text{ for all }n\in\N.
\end{equation}
Fix $0<\vep<\bar{a}$  and set
$\va(x)=\sum_{n=1}^{\infty}(f_n(x+\alpha)-f_n(x))$, where
\[f_n(x)=\frac{q_{n+1}}{q_n^{1+\bar{a}-\vep}}(1-\cos 2\pi
q_nx).\] Note that $f_n(x)\geq 0$ for $x\in\T$ and $n\geq 1$. Let
\[F_n=\bigcup_{j=0}^{q_n-1}\left(\frac{j}{q_n}+\left[-\frac{1}{\sqrt{q_nq_{n+1}}},
\frac{1}{\sqrt{q_nq_{n+1}}}\right]\right)\text{ and }
F=\bigcap_{n=1}^{\infty}F_n.\]

\begin{Lemma}\label{lempoje2}
The function $\va:\T\to\R$ is of class
$C^{\lceil\bar{a}-\vep\rceil}$ and there exist
$\theta=\theta(C_0,d,a,\vep)>0$ and $K=K(a,\vep)>0$ such that
$\va^{(m)}(x)\geq -K$ for all $x\in F$ and $m\in\Z$. If
additionally $q_n^{a^{d-1}}\leq |m|\leq \frac{1}{4}q_{n+1}$ and
$q_n^{\bar{a}}\geq 64 C_0$ then $\va^{(m)}(x)\geq\theta
|m|^{\vep/a^d}-K$.
\end{Lemma}

\begin{proof}
Since
\[f_n(x+\alpha)-f_n(x)=\frac{2q_{n+1}}{q_{n}^{1+\bar{a}-\vep}}\sin\pi
q_n\alpha\cdot\sin 2\pi q_n(x+\alpha/2),\] by (\ref{singora}) and
(\ref{m1}),
\begin{eqnarray*}\left|\frac{d^k}{dx^k}(f_n(x+\alpha)-f_n(x))\right|&=&
\frac{q_{n+1}}{q_{n}^{1+\bar{a}-\vep}}2|\sin\pi
q_n\alpha| \left|\frac{d^k}{dx^k}\sin 2\pi
q_n(x+\alpha/2)\right|\\
&\leq& \frac{q_{n+1}}{q_{n}^{1+\bar{a}-\vep}}2\pi\|q_n\alpha\|
(2\pi q_n )^k\leq \frac{(2\pi)^{k+1}}{q_n^{1+\bar{a}-\vep-k}}.
\end{eqnarray*}
Moreover $\sum_{n=1}^\infty 1/q_n^{1+\bar{a}-\vep-k}<+\infty$ for
$0\leq k\leq \lceil\bar{a}-\vep\rceil$, so $\va\in
C^{\lceil\bar{a}-\vep\rceil}(\T,\R)$.

Suppose that $x\in F$. Then $x\in F_n$ for every $n\geq 1$ and
hence $\|q_nx\|\leq \sqrt{q_n}/\sqrt{q_{n+1}}$.  In view of
(\ref{sindul}),
\[\sin^2\pi q_nx\leq\pi^2\|q_nx\|^2\leq\pi^2q_n/q_{n+1}.\]
Therefore for every $m\in\Z$
\begin{eqnarray*}f_n(x+m\alpha)-f_n(x)&\geq&
-f_n(x)=-\frac{q_{n+1}}{q_n^{1+\bar{a}-\vep}}(1-\cos 2\pi
q_nx)\\&=&-2\frac{q_{n+1}}{q_n^{1+\bar{a}-\vep}}\sin^2 \pi
q_nx\geq
-2\pi^2\frac{q_{n+1}}{q_n^{1+\bar{a}-\vep}}\frac{q_n}{q_{n+1}}
=-\frac{2\pi^2}{q_n^{\bar{a}-\vep}}.
\end{eqnarray*}
Let
\[K=K(a,\vep):=2\pi^2\sum_{n=1}^\infty\frac{1}{4^{n(\bar{a}-\vep)}}=
\frac{2\pi^2}{4^{\bar{a}-\vep}-1}.\]
In view of (\ref{wieloqn}), $q_n\geq 4^n$, so
\begin{equation}\label{ujem}
\va^{(m)}(x)=\sum_{n=1}^\infty(f_n(x+m\alpha)-f_n(x))\geq
-2\pi^2\sum_{n=1}^\infty\frac{1}{q_n^{\bar{a}-\vep}}\geq-K.
\end{equation}

Suppose additionally that $q_n^{a^{d-1}}\leq |m|\leq
\frac{1}{4}q_{n+1}$. By (\ref{m1}), $$\|mq_n\alpha\|\leq
|m|\|q_n\alpha\|<1/4,\text{ and hence }\|mq_n\alpha\|=
|m|\|q_n\alpha\|.$$ Therefore
\begin{eqnarray*}
f_n(x+m\alpha)-f_n(x)&=&\frac{q_{n+1}}{q_{n}^{1+\bar{a}-\vep}}(\cos2\pi
q_nx-\cos2\pi
q_n(x+m\alpha))\\&=&\frac{q_{n+1}}{q_{n}^{1+\bar{a}-\vep}}\left(\cos2\pi(\pm\|
q_nx\|)-\cos2\pi
(\pm\|q_nx\|+\|mq_n\alpha\|)\right)\\
&=&\frac{q_{n+1}}{q_{n}^{1+\bar{a}-\vep}}2\sin\pi (\pm2\|q_nx\|+
\|mq_n\alpha\|)\sin\pi
\|mq_n\alpha\|\\
&=&\frac{q_{n+1}}{q_{n}^{1+\bar{a}-\vep}}2\sin\pi (\pm2\|q_nx\|+
|m|\|q_n\alpha\|)\sin\pi
|m|\|q_n\alpha\|.\\
\end{eqnarray*}
By (\ref{wieloqn}),
\[\sqrt{q_nq_{n+1}}\leq\sqrt{C_0}q_n^{\frac{1+a^d}{2}}=
\sqrt{C_0}q_n^{a^{d-1}-\bar{a}/2}\leq \sqrt{\frac{C_0}{q_n^{\bar{a}}}}|m|
\leq\frac{|m|}{8},\] whenever $q_n^{\bar{a}}\geq 64 C_0$.
Moreover, by (\ref{m1}), $\frac{|m|}{2q_{n+1}}\leq
|m|\|q_n\alpha\|<1/4$ and hence
\[\pm2\|q_nx\|+|m|\|q_n\alpha\|\geq |m|\|q_n\alpha\|-2\|q_nx\|
\geq\frac{|m|}{2q_{n+1}}-2\frac{\sqrt{q_nq_{n+1}}}{q_{n+1}}\geq \frac{|m|}{4q_{n+1}}.\]
Similarly
\[\pm2\|q_nx\|+|m|\|q_n\alpha\|\leq |m|\|q_n\alpha\|+
2\|q_nx\|\leq\frac{|m|}{q_{n+1}}+2\frac{\sqrt{q_nq_{n+1}}}
{q_{n+1}}\leq\frac{2|m|}{q_{n+1}}\leq\frac{1}{2}.\]
By (\ref{sindul}), it follows that
\[\sin\pi |m|\|q_n\alpha\|\geq 2 |m|\|q_n\alpha\|\geq\frac{|m|}{q_{n+1}}\]
and
\[\sin\pi(\pm2\|q_nx\|+|m|\|q_n\alpha\|)\geq 2(\pm2\|q_nx\|+|m|\|q_n\alpha\|)
\geq\frac{|m|}{2q_{n+1}}.\]
Thus
\begin{eqnarray*}
f_n(x+m\alpha)-f_n(x)
&=&\frac{2q_{n+1}}{q_{n}^{1+\bar{a}-\vep}}\sin\pi (\pm2\|q_nx\|+
|m|\|q_n\alpha\|)\sin\pi |m|\|q_n\alpha\| \\&\geq&
\frac{m^2}{q_{n}^{1+\bar{a}-\vep}q_{n+1}} \geq
\frac{q_n^{2a^{d-1}}}{C_0q_{n}^{1+\bar{a}-\vep}q^{a^d}_{n}}=\frac{q_n^{\vep}}{C_0}
\geq\frac{q_{n+1}^{\vep/a^d}}{C_0^{1+\frac{\vep}{a^d}}}
\geq\frac{|m|^{\vep/a^d}}{C_0^{1+\frac{\vep}{a^d}}}.
\end{eqnarray*}
In view of the proof of (\ref{ujem}), we conclude that
\begin{equation*}\va^{(m)}(x)=f_n(x+m\alpha)-f_n(x)+\sum_{l\neq
n}^\infty(f_l(x+m\alpha)-f_l(x))\geq
\frac{|m|^{\vep/a^d}}{C_0^{1+\frac{\vep}{a^d}}}-K.
\end{equation*}
\end{proof}

Set $C=2\lceil 4^{a+2}\rceil$. Let $\alpha_1,\ldots,\alpha_d$ be
irrational numbers. Let $(q_n^{(j)})_{n=1}^{\infty}$ stand for the
sequence of denominators of $\alpha_j$ for $j=1,\ldots,d$. Assume
that the denominators of  $\alpha_1,\ldots,\alpha_d$ satisfy the
following inequalities
\begin{equation}\label{zalalpha}
4(q_n^{(j)})^a\leq q_n^{(j+1)}\leq C(q_n^{(j)})^a\text{ for all
}1\leq j\leq d\text{ and }n\geq 1,
\end{equation}
in which we use the notation $q_n^{(d+1)}=q_{n+1}^{(1)}$ and
$q_n^{(0)}=q_{n-1}^{(d)}$. It is easy to see that
\begin{equation}\label{zgor1}
4(q_n^{(j)})^{a^{d-1}}\leq
4^{1+\ldots+a^{d-2}}(q_n^{(j)})^{a^{d-1}}\leq q_{n+1}^{(j-1)}
\end{equation}
and
\begin{equation}\label{zdol1}
4(q_n^{(j)})^{a^d}\leq q^{(j)}_{n+1}\leq
C^{1+a+\ldots+a^{d-1}}(q_n^{(j)})^{a^d}=C_a(q_n^{(j)})^{a^d}
\end{equation}
for all $1\leq j\leq d$ and $n\in\N$.

\begin{Remark} Denote by $\mathfrak{C}(a)$ the set of all
$(\alpha_1,\ldots,\alpha_d)\in\T^d$ satisfying (\ref{zalalpha}).
In Appendix~\ref{dodatek}  we will show that for uncountably many
$(\alpha_1,\ldots,\alpha_d)\in\mathfrak{C}(a)$ the rotation on
$\T^d$ by the vector $(\alpha_1,\ldots,\alpha_d)$ is minimal.
\end{Remark}

Take $(\alpha_1,\ldots,\alpha_d)\in\mathfrak{C}(a)$ such that the
rotation $T:\T^d\to\T^d$ given by
$$T(x_1,\ldots,x_d)=(x_1+\alpha_1,\ldots,x_d+\alpha_d)$$  is
minimal. Let $0<\vep<\bar{a}$ and set
$\va_j(x)=\sum_{n=1}^{\infty}(f_{n,j}(x+\alpha_j)-f_{n,j}(x))$,
where
\[f_{n,j}(x)=\frac{q_{n+1}^{(j)}}{(q_n^{(j)})^{1+\bar{a}-\vep}}(1-\cos 2\pi
q_n^{(j)}x)\text{ for }j=1,\ldots,d.\] Let us consider  the
function $\va:\T^d\to\R$,
\[\va(x_1,\ldots,x_d)=\va_1(x_1)+\ldots+\va_d(x_d).\]
For $1\leq j\leq d$ set
\[F^{(j)}=\bigcap_{n=1}^{\infty}\bigcup_{l=0}^{q^{(j)}_n-1}
\left(\frac{l}{q^{(j)}_n}+\left[-\frac{1}
{\sqrt{q^{(j)}_{n}q^{(j)}_{n+1}}},\frac{1}
{\sqrt{q^{(j)}_{n}q^{(j)}_{n+1}}}\right]\right).\]

\begin{Th}
The function $\va:\T^d\to\R$ is of class
$C^{\lceil\bar{a}-\vep\rceil}$ and
\[\va^{(m)}(x_1,\ldots,x_d)\to+\infty\text{ as }|m|\to+\infty\]
for each $(x_1,\ldots,x_d)\in F^{(1)}\times\ldots\times F^{(d)}$.
\end{Th}

\begin{proof}
For each integer $m$ with $|m|\geq q_1^{(d)}/4$ there exist
$n\in\N$ and $1\leq j\leq d$ such that
$\frac{1}{4}q_{n+1}^{(j-1)}\leq |m|\leq \frac{1}{4}q_{n+1}^{(j)}$.
Note that if $|m|\to+\infty$ then $n\to+\infty$, so
$(q_{n}^{(j)})^{\overline{a}}\geq 64 C_a$ whenever $|m|$ is large
enough. From~(\ref{zgor1}),
\[(q_n^{(j)})^{a^{d-1}}\leq\frac{1}{4}q_{n+1}^{(j-1)}
\leq|m|\leq \frac{1}{4}q_{n+1}^{(j)}.\]
By (\ref{zdol1}) and Lemma~\ref{lempoje2}, for every
$(x_1,\ldots,x_d)\in F^{(1)}\times\ldots\times F^{(d)}$ we have
\begin{eqnarray*}\va^{(m)}(x_1,\ldots,x_d)&=&
\sum_{k=1}^d\va_k^{(m)}(x_k)\geq\va_j^{(m)}(x_j)-(d-1)K(a,\vep)\\
&\geq&\theta(C_a,d,a,\vep)|m|^{\vep/a^d}-dK(a,\vep).
\end{eqnarray*}
Consequently,
\[\va^{(m)}(x_1,\ldots,x_d)\to+\infty\text{ as }|m|\to+\infty.\]
\end{proof}

\begin{Prop} For every $1\leq j\leq d$ we have $\dim_HF^{(j)}\geq
\frac{1}{1+a^d}$, in particular
$\dim_H(F^{(1)}\times\ldots\times F^{(d)})\geq \frac{d}{1+a^d}$.
\end{Prop}

\begin{proof}
Fix $1\leq j\leq d$ and we will write $q_n$ instead of
$q_n^{(j)}$. Note that to calculate the Hausdorff dimension of
$F^{(j)}$ we can use the scheme presented at the beginning of
Section~\ref{secdimension} and Remark~\ref{liczint} in which
\[a=\frac{2}{\sqrt{q_nq_{n+1}}},\;\;b=\frac{2}{\sqrt{q_{n-1}q_{n}}}\;\;
\text{ and }\;\;h=\frac{1}{q_n}.\] Here
\[\vep_n=\frac{1}{q_n}-\frac{2}{\sqrt{q_nq_{n+1}}}\geq\frac{1}{2q_n}\;\;\text{ and
}\;\;
m_n\geq\frac{\frac{2}{\sqrt{q_{n-1}q_n}}}{4\frac{1}{q_n}}=
\frac{\sqrt{q_n}}{2\sqrt{q_{n-1}}}.\]
Thus
\[m_1\ldots m_{n-1}\geq\frac{\sqrt{q_{n-1}}}{2^{n-1}}\] and using (\ref{zdol1})
\[m_n\vep_n\geq\frac{1}{4\sqrt{q_{n-1}q_n}}\geq
\frac{1}{4\sqrt{C_a}q_{n-1}^{\frac{1+a^d}{2}}}.\] Moreover,
$q_{n+1}\geq q_n^{a^d}$, and hence (reminding that $q_1\geq 4
$)\[\log q_n\geq\log q_1^{a^{d(n-1)}}\geq a^{d(n-1)}\log 4.\] By
Proposition~\ref{propfal}, it follows that
\[\dim_H{F^{(j)}}\geq\liminf_{n\to\infty}\frac{\frac{1}{2}
\log q_{n-1}-(n-1)\log 2}{\frac{1+a^d}{2}\log q_{n-1}+\log 4+
\log\sqrt{C_a}}=\frac{1}{1+a^d}.\]
Since
\[\dim_H(F^{(1)}\times\ldots\times F^{(d)})\geq\dim_H{F^{(1)}}+
\ldots+\dim_H{F^{(d)}},\]
the proof is complete.
\end{proof}

\begin{Remark}For every $d\geq 3$  in order to obtain the maximum of
$\bar{a}$ we have to  choose of $a=2(d-1)/d$. Then
\[\bar{a}=\frac{2^d}{d}\left(\frac{d-1}{d}\right)^{d-1}-1>\frac{2^d}{ed}-1.\]
In particular, for $d=3$ we have
$\bar{a}={8}/{3}\cdot\left({2}/{3}\right)^{2}=32/27<2$, so we can
only choose $\vep>0$ so that $\lceil\bar{a}-\vep\rceil=1$.
However, if $d\to+\infty$ then the degree of smoothness of
$\varphi$ grows exponentially.
\end{Remark}

\section{Differential equations and flows}\label{difeq}
Consider now a system of differential equations on
$\T^d\times\T\times\R$
\begin{equation}\label{system}
\begin{array}{lll}
\frac{d\ov{x}}{dt} &= &\ov{\alpha}\\
\frac{dz}{dt} &=&1\\
\frac{dy}{dt} &=& f(\ov{x},z),
\end{array}\end{equation}
where $\ov{\alpha}=(\alpha_1,\ldots,\alpha_d)$ induces a minimal
rotation $T$ on $\T^d$ and $f:\T^d\times\T\to\R$ is of class $C^r$
for some $r\geq0$. Denote by $(\Phi_t)_{t\in\R}$ the corresponding
flow on $\T^{d+1}\times\R$
$$
\Phi_t(\ov{x},z,y)=\left(x+t\ov{\alpha},z+t,\int_0^tf(\ov{x}+s\ov{\alpha},z+s)\,ds\right).$$
Then $\T^d\times\{0\}\times\R=\T^d\times\R$ is a global section
for $(\Phi_t)_{t\in\R}$ and the Poincar\'e map is given by the
formula
$$
(\ov{x},y)\mapsto (\ov{x}+\ov{\alpha},y+\va(\ov{x})),$$ where
\begin{equation}\label{system1}
\va(\ov{x})=\int_0^1 f(\ov{x}+s\ov{\alpha},s)\,ds,\end{equation}
and hence $\va:\T^d\to\R$ is also of $C^r$ class.  Reciprocally,
if $\va:\T^d\to\R$ is of $C^r$ class then we can find
$f:\T^d\times\T\to\R$ which is of $C^r$ class so
that~(\ref{system1}) holds (for example
$f(\overline{x},z)=\va(\ov{x}-z\ov{\alpha})b(z)$ will do provided
$b:[0,1]\to\R$ is smooth, $\int_0^1b(t)\,dt=1$ and
$b|_{[0,\eta]}=0=b|_{[1-\eta,1]}$ for some $0<\eta<1/2$).

Now the flow $(\Phi_t)_{t\in\R}$ is topologically the same as the
suspension flow over $T_\va$. In particular, closed orbits of
$(\Phi_t)_{t\in\R}$ which we may identify with closed orbits of
the suspension flow are in a natural correspondence with closed
orbits of $T_\va$. More generally,  each minimal subset for the
suspension flow is of the form $M\times[0,1)$ where
$M\subset\T^d\times\R$ is a minimal subset for $T_\va$. We will
now make use of our knowledge about properties of $T_\va$ to
derive properties of $(\Phi_t)_{t\in\R}$.

First of all, we note that $(\Phi_t)_{t\in\R}$ is never minimal.
Then, note that
\begin{equation}\label{calka}
\int_{\T^d}
\va(\ov{x})\,d\ov{x}=\int_{\T^d\times\T}f(\ov{x},z)\,d\ov{x}dz,
\end{equation}
so we should constantly assume that the latter integral vanishes;
otherwise $\T^d\times\T\times\R$ is foliated into closed orbits of
$(\Phi_t)_{t\in\R}$ (each orbit being homeomorphic to $\R$). When
the integral vanishes and $\va(\ov{x})=j(\ov{x})-j(T\ov{x})$ for a
continuous $j:\T^d\to\R$ then again $\T^d\times\T\times\R$ is
foliated into minimal components of $(\Phi_t)_{t\in\R}$, however
now each minimal component is compact. This situation is
equivalent to saying that there exists an orbit of
$(\Phi_t)_{t\in\R}$ which is relatively compact (and then all
orbits are relatively compact). Finally, if the integral
(\ref{calka}) vanishes and no orbit of $(\Phi_t)_{t\in\R}$ is
relatively compact then $(\Phi_t)_{t\in\R}$ is topologically
transitive, that is, there is a dense orbit (and since it is not
minimal there are orbits which are not dense). A natural question
arises to decide whether in the transitive and smooth ($r\geq 1$)
case closed orbits can exist. The answer is negative if $d=1$
 (in fact, for $d=1$ the flow corresponding
to~(\ref{system}) has no minimal subset). We can now interpret the
results of Section~\ref{highdim} as the positive answer to the
question in case $d\geq3$ (although with some restriction on the
degree of smoothness of the vector field in~(\ref{system})).

\subsection{Differential equations on $\R^3$}\label{sectr3} For any
irrational number $\alpha>0$ let us consider the system of
differential equations on $\R^3$
\begin{align}\label{rowr3bezper}
\begin{split}
x'&=-2\pi y+2\pi \alpha xz\\
y'&=2\pi x+2\pi \alpha yz\\
z'&=\pi \alpha(1-x^2-y^2+z^2).
\end{split}
\end{align}
It is easily checked that
$\ln((\sqrt{x^2+y^2}+1)^2+z^2)-\ln((\sqrt{x^2+y^2}-1)^2+z^2)$ is a
first integral of (\ref{rowr3bezper}). For $a>0$ the set of
$(x,y,z)\in\R^3$ satisfying
\begin{equation}\label{torusy}
(\sqrt{x^2+y^2}+1)^2+z^2=a((\sqrt{x^2+y^2}-1)^2+z^2),
\end{equation}
is a torus; indeed, by setting $r=\sqrt{x^2+y^2}$, (\ref{torusy})
is equivalent to
\[\left(r-\frac{a+1}{a-1}\right)^2+z^2=\frac{4a}{(a-1)^2}.\]
It follows that the corresponding family of tori establishes an
invariant foliation of $\R^3\setminus (\s_0\cup\R_0)$, where
\[\s_0=\{(x,y,z):x^2+y^2=1,z=0\}\text{ and }\R_0=\{(x,y,z):x=y=0\}.\]
Moreover, the flow corresponding to (\ref{rowr3bezper}) acts on
each invariant torus as the linear flow in the direction
$(\alpha,1)$ -- we will see a relevant computation in a perturbed
situation in a while.

 The aim of this section is to give (for every
irrational $\alpha$) a continuous perturbation of
(\ref{rowr3bezper}) which completely destroys its integrable
dynamics.  We will show that under some special continuous
perturbation of (\ref{rowr3bezper})
 the resulting systems  have plenty of orbits
which are dense, homoclinic and heteroclinic to limit cycles. The
class of considered perturbations is in the spirit of Chapter~XIX
in \cite{Po}.

 Let $\psi:\T^2\to\R$ be a continuous
function. Set
\[\omega(x,y,z)=\frac{\arg(x^2+y^2+z^2-1-2zi)}{2\pi}\text{ and }\theta(x,y)=\frac{\arg(x+yi)}{2\pi}\]
for every $\R^3\setminus(\s_0\cup\R_0)$ and  denote by
$F:\R^3\setminus(\s_0\cup\R_0)\to\R$ the continuous function
\[F(x,y,z)=\ln\sqrt{\frac{(\sqrt{x^2+y^2}+1)^2+z^2}{(\sqrt{x^2+y^2}-1)^2+z^2}}\cdot\psi\left(\omega(x,y,z),\theta(x,y)\right).\]
We will deal with the perturbed  differential equation
\begin{align}\label{rowr3}
\begin{split}
x'&=-2\pi y+2\pi \alpha xz+
\frac{(1-x^2-y^2+z^2)x}{2\sqrt{x^2+y^2}}F(x,y,z)\\
y'&=2\pi x+2\pi \alpha yz+
\frac{(1-x^2-y^2+z^2)y}{2\sqrt{x^2+y^2}}F(x,y,z)\\
z'&=\pi \alpha(1-x^2-y^2+z^2)-z\sqrt{x^2+y^2}F(x,y,z).
\end{split}
\end{align} It is easily checked that the right hand side of
(\ref{rowr3}) can be continuously extended to $\R^3$ (indeed,
since $|1-r^2+z^2|\leq 3\sqrt{(r-1)^2+z^2}$ for $|r-1|$ and $|z|$
sufficiently small, $(1-r^2+z^2)\ln\sqrt{(r-1)^2+z^2}\to 0$ as
$r\to 1$, $z\to 0$ and therefore on  $\s_0\cup\R_0$ we come back
to  (\ref{rowr3bezper})), so the equation (\ref{rowr3}) is well
defined on $\R^3$. The existence and uniqueness of solutions
(\ref{rowr3}) we will prove later by applying a change of
coordinates (solutions are defined for all $t\in\R$, except for
the line $\R_0$). Denote by $(\Psi_t)_{t\in\R}$ the flow
corresponding to (\ref{rowr3}). We have already noticed that
$x'=-2\pi y,\;y'=2\pi x,z'=0$ on $\s_0$ and $x'=0,\;y'=0,z'=\pi
\alpha(1+z^2)$ on $\R_0$. Therefore $\s_0$ is
$(\Psi_t)_{t\in\R}$-invariant  and
$\Psi_t(x,y,0)=\Psi_t(x+iy,0)=(e^{2\pi it}(x+iy),0)$ on $\s_0$.
Moreover, $(\Psi_t)_{t\in\R}$ is a local flow on $\R_0$ and
\[\Psi_t(0,0,z)=(0,0,\tan(\pi\alpha t+\arctan z))\text{ for }t\in \left(-\frac{1}{2\alpha},\frac{1}{2\alpha}\right)-\frac{\arctan z}{\pi\alpha}.\]

We will show that $(\Psi_t)_{t\in\R}$ acts on
$\R^3\setminus(\s_0\cup\R_0)$ indeed as a flow. We will use a
hyperbolic polar coordinates on the hyperbolic half-plane
$\{(z,r):r>0,z\in\R\}$ given by
\[z+ir\mapsto\frac{i(z+ir)+1}{(z+ir)+i}=ie^{-e^s}e^{2\pi i\omega}\]
together with the usual polar coordinates $x=r\cos2\pi\theta$, $
y=r\sin2\pi\theta$. It results in  toral coordinates
$(\omega,\theta,s)\in\T\times\T\times\R$ of
$\R^3\setminus(\s_0\cup\R_0)$ given by
\begin{align*}
x&=\left(\frac{-2(e^{-e^s}\cos2\pi\omega-1)}{(e^{-e^s}\sin2\pi\omega)^2+(e^{-e^s}\cos2\pi\omega-1)^2}-1\right)\cos2\pi\theta,\\
y&=\left(\frac{-2(e^{-e^s}\cos2\pi\omega-1)}{(e^{-e^s}\sin2\pi\omega)^2+(e^{-e^s}\cos2\pi\omega-1)^2}-1\right)\sin2\pi\theta,\\
z&=\frac{-2e^{-e^s}\sin2\pi\omega}{(e^{-e^s}\sin2\pi\omega)^2+(e^{-e^s}\cos2\pi\omega-1)^2}.\\
\end{align*}
Denote by $\Upsilon:\T\times\T\times\R\to
\R^3\setminus(\s_0\cup\R_0)$ the map establishing the change of
coordinates. The inverse change of coordinates is given by
\begin{align*}
2\pi\omega&=\arg(x^2+y^2+z^2-1-2zi),\\
2\pi\theta&=\arg(x+yi),\\
s&=\ln\ln\sqrt{\frac{(\sqrt{x^2+y^2}+1)^2+z^2}{(\sqrt{x^2+y^2}
-1)^2+z^2}};
\end{align*}
indeed,
\[e^{-e^s}=\left|\frac{i(z+ir)+1}{(z+ir)+i}\right|=\sqrt{\frac{(r-1)^2+z^2}{(r+1)^2+z^2}},\]
\[2\pi\omega=\arg\frac{z+i(r-1)}{z+(r+1)}=\arg(r^2+z^2-1-2zi).\]
 Setting additionally $u=e^s$  we have
\begin{align}
x'&=-2\pi y+2\pi \alpha xz+\frac{(1-r^2+z^2)x}{2r}u
\psi\left(\omega,\theta\right),\notag\\
y'&=2\pi x+2\pi \alpha yz+
\frac{(1-r^2+z^2)y}{2r}u\psi\left(\omega,\theta\right),\notag\\
\label{pochz}z'&=\pi\alpha(1-r^2+z^2)-zru\psi\left(\omega,\theta\right).
\end{align}
It follows that
\begin{equation}\label{pochr}
r'=\frac{xx'+yy'}{r}=2\pi\alpha rz+\frac{1-r^2+z^2}{2}u
\psi\left(\omega,\theta\right),
\end{equation}
and
\[xy'-yx'=2\pi(x^2+y^2).\]
Since $x=r\cos2\pi\theta$, $y=r\sin2\pi\theta$, we have
\begin{equation}\label{pochtheta}
2\pi\theta'= \frac{x}{r}\left(\frac{y}{r}\right)'-
\frac{y}{r}\left(\frac{x}{r}\right)'
=\frac{xy'-yx'}{x^2+y^2}=2\pi.
\end{equation}
Next note that
\[|x^2+y^2+z^2-1-2zi|^2=(r^2+z^2-1)^2+(2z)^2=((r+1)^2+z^2)((r-1)^2+z^2),\]
and hence
$$r^2+z^2-1=\sqrt{((r+1)^2+z^2)((r-1)^2+z^2)}\cos2\pi\omega,$$
$$-2z=\sqrt{((r+1)^2+z^2)((r-1)^2+z^2)}\sin2\pi\omega.$$ It follows
that (similarly as in (\ref{pochtheta}))
\[2\pi\omega'=\frac{(-2z)'(r^2+z^2-1)-(r^2+z^2-1)'(-2z)}{((r+1)^2+z^2)((r-1)^2+z^2)}.\]
Moreover, by (\ref{pochz}) and (\ref{pochr}),
\begin{multline*}
(r^2+z^2-1)'2z-(2z)'(r^2+z^2-1)=2(2zrr'+(1-r^2+z^2)z')\\
=2\pi\alpha(4
z^2r^2+(1-r^2+z^2)^2)=2\pi\alpha((r+1)^2+z^2)((r-1)^2+z^2),
\end{multline*}
so $2\pi\omega'=2\pi\alpha$. Since
$u=\frac{1}{2}\left(\ln((r+1)^2+z^2)-\ln((r-1)^2+z^2)\right)$, we
have
\[u'=\frac{r'(r+1)+z'z}{(r+1)^2+z^2}-\frac{r'(r-1)+z'z}{(r-1)^2+z^2}=
\frac{2(1-r^2+z^2)r'-4rzz'}{((r+1)^2+z^2)((r-1)^2+z^2)}.\]
Moreover, by (\ref{pochz}) and (\ref{pochr}),
\begin{align*}2(1-r^2+z^2)r'-4rzz'=&((z^2-r^2+1)^2+4(rz)^2)u\psi\left(\omega,\theta\right)\\=&
((r+1)^2+z^2)((r-1)^2+z^2)u\psi\left(\omega,\theta\right),\end{align*}
hence $u'=u\psi\left(\omega,\theta\right)$. Therefore
\[s'=u'/u=\psi\left(\omega,\theta\right),\]
hence in the new coordinates the  differential equation
(\ref{rowr3}) takes the form
\[\omega'=\alpha,\;\;\theta'=1,\;\;s'=\psi\left(\omega,\theta\right)\]
and we return to the scheme from  Section~\ref{difeq}. Denote by
$(\Phi_t)_{t\in\R}$ the corresponding flow on
$\T\times\T\times\R$, then
\[\Phi_t(\omega,\theta,s)=(\omega+t\alpha,\theta+t,s+\int_0^t\psi(\omega+\tau\alpha,\theta+\tau)d\tau).\]

Let $\alpha$ be an arbitrary irrational number and let
$\va:\T\to\R$ stand for the function constructed in
Section~\ref{varia}. Choose a continuous function $\psi:\T^2\to\R$
for which $\va(\omega)=\int_0^1\psi(\omega+\tau\alpha,\tau)d\tau$.
Assume that for $(\omega,\theta,s)\in\T\times\T\times\R$ we have
$\omega-\alpha\theta\in D^{s_-+}$. Then
$\va^{(n)}(\omega-\alpha\theta)\to+\infty$ as $n\to+\infty$, so
\begin{align*}s_n:&=s+\int_0^{n}\psi(\omega+\tau\alpha,\theta+\tau)d\tau\\
&=s+\va^{(n)}(\omega-\alpha\theta)+\int_0^{\theta}\left(\psi(\omega+(\tau+n-\theta)\alpha,\tau)-\psi(\omega+(\tau-\theta)\alpha,\tau)\right)d\tau\\
&\to+\infty.
\end{align*}
Note that if $(x,y,z)=\Upsilon(\omega,\theta,s)$ then
\begin{align*}
x+iy&=\frac{1-e^{-2e^s}}{e^{-2e^s}-2e^{-e^s}\cos2\pi\omega+1}e^{2\pi
i \theta},\;
&z=\frac{-2e^{-e^s}\sin2\pi\omega}{{e^{-2e^s}-2e^{-e^s}\cos2\pi\omega+1}}.\\
\end{align*}
Since $e^{-2e^{s_n}}\to 0$, setting $\omega_n=\omega+n\alpha$ we
obtain
\begin{align*}\Psi_n(x,y,z)&=\Psi_n\circ\Upsilon(\omega,\theta,s)=\Upsilon\circ\Phi_n(\omega,\theta,s)=\Upsilon(\omega_n,\theta+n,s_n)=\Upsilon(\omega_n,\theta,s_n)
\\&=\left(\frac{(1-e^{-2e^{s_n}})e^{2\pi
i\theta}}{e^{-2e^{s_n}}-2e^{-e^{s_n}} \cos 2\pi\omega_n+1},
\frac{-2e^{-e^{s_n}}\sin2\pi\omega_n}{{e^{-2e^{s_n}}-2e^{-e^{s_n}}
\cos2\pi\omega_n+1}}\right)\\ &\to(e^{2\pi i\theta},0)\in\s_0.
\end{align*}
It follows that the $\omega$--limit set of $(x,y,z)$ is equal to
$\s_0$.

Now assume that $\omega-\alpha\theta\in D^{s_--}$ and $\omega\neq
0$. Then
\begin{align*}s_n:&=s+\int_0^{n/\alpha}\psi(\omega+\tau\alpha,\theta+\tau)d\tau\\
&=s+\va^{(\lfloor\theta+
n/\alpha\rfloor)}(\omega-\alpha\theta)\\
&\quad+\int_0^{\{\theta+n/\alpha\}}\psi\left(\omega+(\tau-\theta+\lfloor\theta+
n/\alpha\rfloor)\alpha,\tau\right)d\tau-\int_0^{\theta}\psi\left(\omega+(\tau-\theta)\alpha,\tau\right)d\tau\\
&\to-\infty
\end{align*}
as $n\to+\infty$, so $e^{-2e^{s_n}}\to 1$. Setting
$\theta_n=\theta+n/\alpha$ we obtain
\begin{align*}\Psi_{n/\alpha}(x,y,z)
&=\Upsilon\circ\Phi_{n/\alpha}(\omega,\theta,s)=\Upsilon(\omega+n,\theta_n,s_n)=\Upsilon(\omega,\theta_n,s_n)\\
&=\left(\frac{(1-e^{-2e^{s_n}})e^{2\pi
i\theta_n}}{e^{-2e^{s_n}}-2e^{-e^{s_n}} \cos 2\pi\omega+1},
\frac{-2e^{-e^{s_n}}\sin2\pi\omega}{{e^{-2e^{s_n}}-2e^{-e^{s_n}}
\cos2\pi\omega+1}}\right)\\
&\to\left(0,\frac{-\sin2\pi\omega}{1-\cos2\pi\omega}\right)=(0,\tan(\pi\omega-\pi/2))\in\R_0.
\end{align*}It follows that the $\omega$--limit set of $(x,y,z)$ is equal to
$\R_0$.

Let $A_+:=\s_0$ and $A_-:=\R_0$. Similar arguments to those above
show that if $\omega-\alpha\theta\in D^{s_-s_+}$ then the
$\alpha$--limit set of $(x,y,z)$ is $A_{s_-}$. Recall that the set
$D^{s_-s_+}$ is invariant under the rotation by $\alpha$, so  it
is dense. Thus
\[\{\Upsilon(\omega,\theta,s):(\omega,\theta,s)\in\T\times\T\times\R,\,\omega-\alpha\theta\in D^{s_-s_+}\}\]
is dense in $\R^3$. Moreover, the Hausdorff dimension of this set
is no smaller than $\dim_HD_{s_-s_+}+2$. In view of
Theorem~\ref{wym12} we have the following.

\begin{Th}For every irrational $\alpha$ there exists a continuous
function $\psi:\T^2\to\R$ such that the flow $(\Psi_t)_{t\in\R}$
corresponding to (\ref{rowr3}) is transitive and for each
$s_-,s_+\in\{-,+\}$ the set of points such that the
$\omega$--limit set is equal to $A_{s_+}$ and the $\alpha$--limit
set is equal to $A_{s_-}$ is dense. Moreover, for almost every
$\alpha$ the Hausdorff dimension of each such set is no smaller
than $5/2$.
\end{Th}

\subsection{Differential equations on $\s^3$}\label{sects3} In this
section we will deal with  continuous (or even H\"older
continuous) perturbations of the completely integrable system
$z'_1=iaz_1$, $z'_2=ibz_2$ ($a,b\in\R\setminus\{0\}$) on
$\C\times\C\simeq\R^4$.  Let us consider the system of
differential equations
\begin{align}\label{rowzesp}
\begin{split}
z_1'&=iaz_1-F(z_1,z_2)\overline{z_2}\\
z_2'&=ibz_2+F(z_1,z_2)\overline{z_1},
\end{split}
\end{align} where
$F:\C\times\C\to\C$ is a continuous function $\R_+$-homogeneous of
degree $1$ on each coordinate, i.e.\
$F(t_1z_1,t_2z_2)=t_1t_2F(z_1,z_2)$ for all $t_1,t_2>0$, and such
that $\overline{z_1}\overline{z_2}F(z_1,z_2)\in\R$. Denote by
$(\Psi_t)_{t\in\R}$ the associated flow on $\C\times\C$. The
existence of $(\Psi_t)_{t\in\R}$ will be shown as a byproduct.
Note that $|z_1|^2+|z_2|^2$ is a first integral for
(\ref{rowzesp}). Indeed,
\begin{equation}\label{pochz1}
\frac{d}{dt}|z_1|^2=2\Re
z_1'\overline{z_1}=2\Re(ia|z_1|^2-F(z_1,z_2)\overline{z_1}\overline{z_2})=-2F(z_1,z_2)\overline{z_1}\overline{z_2}
\end{equation}
and
\begin{equation}\label{pochz2}
\frac{d}{dt}|z_2|^2=2\Re
z_2'\overline{z_2}=2\Re(ia|z_2|^2+F(z_1,z_2)\overline{z_1}\overline{z_2})=2F(z_1,z_2)\overline{z_1}\overline{z_2},
\end{equation}
so $\frac{d}{dt}(|z_1|^2+|z_2|^2)=1$. Therefore,
\[\s^3=\{(z_1,z_2)\in\C\times\C:|z_1|^2+|z_2|^2=1\}\]
is $(\Phi_t)$--invariant and we confine ourselves to the study of
$(\Psi_t)_{t\in\R}$ on  $\s^3$. Let us consider
$$\widetilde{F}:\C\setminus\{0\}\times\C\setminus\{0\}\to\R,\;\;
\widetilde{F}(z_1,z_2)=\frac{F(z_1,z_2)}{z_1z_2}=\frac{F(z_1,z_2)\overline{z_1}\overline{z_2}}{|z_1|^2|z_2|^2}\in\R.$$
By assumptions, $\widetilde{F}$ is continuous and
$\widetilde{F}(z_1,z_2)=\widetilde{F}(z_1/|z_1|,z_2/|z_2|)$. Let
\[\psi:\T\times\T\to\R,\;\;\psi(\omega,\theta)=\widetilde{F}(e^{2\pi i\omega},e^{2\pi i\theta}),\]
so $\psi$ is continuous. Moreover, note that each continuous
function $\psi:\T\times\T\to\R$ determines a  continuous function
 $F:\C\times\C\to\C$ which is $\R_+$-homogeneous   of degree $1$ on each
coordinate and such that
$\overline{z_1}\overline{z_2}F(z_1,z_2)\in\R$ as follows
\[F(z_1,z_2)=z_1z_2\psi(\omega,\theta)\text{ whenever }z_1=|z_1|e^{2\pi i\omega},\;z_2=|z_2|e^{2\pi i\theta}.\]
If additionally $\psi$ is $\gamma$-H\"older continuous then
$F:\s^3\to\C$ is $\gamma$-H\"older continuous as well (see
Proposition~\ref{holdwu} in Appendix~\ref{sechol}). Let
\[\s^1_-:=\{(z_1,z_2)\in\s^3:|z_1|=1,z_2=0\}\text{ and }\s^1_+:=\{(z_1,z_2)\in\s^3:z_1=0,|z_2|=1\}.\]
Observe that $\s^1_+$, $\s^1_-$ are invariant sets and
\[\Psi(z_1,0)=(e^{iat}z_1,0)\text{ and }\Psi(0,z_2)=(0,e^{ibt}z_2).\]
Let us consider a new coordinates
$(\omega,\theta,s)\in\T\times\T\times\R$ of
$\s^3\setminus(\s^1_+\cup\s^1_-)$ given by
\[(z_1,z_2)=\Upsilon(\omega,\theta,s)=(e^{2\pi i\omega}\cos\arctan e^s,e^{2\pi i\theta}\sin\arctan e^s).\]
Then
\begin{equation*}
e^{2\pi i\omega}=z_1/|z_1|,\;e^{2\pi i\theta}=z_2/|z_2|\text{ and
}|z_1|=\cos\arctan e^s,\;\;|z_2|=\sin\arctan e^s.
\end{equation*}
Thus
\begin{equation*}
\frac{|z_2|}{|z_1|}=\frac{\sin\arctan e^s}{\cos\arctan e^s}=e^s,
\end{equation*}
so
$$s=\ln|z_2|-\ln|z_1|=\frac{1}{2}(\ln|z_2|^2-\ln|z_1|^2).$$ By
(\ref{pochz1}) and (\ref{pochz2}), it follows that
\begin{align*}
\begin{split}
\frac{ds}{dt}&=\frac{1}{2}\left(\frac{1}{|z_2|^2}\frac{d}{dt}|z_2|^2-\frac{1}{|z_1|^2}\frac{d}{dt}|z_1|^2\right)=
F(z_1,z_2)\overline{z_1}\overline{z_2}\left(\frac{1}{|z_2|^2}+\frac{1}{|z_1|^2}\right)\\
&=F(z_1,z_2)\frac{\overline{z_1}\overline{z_2}}{|z_1|^2|z_2|^2}=\widetilde{F}(z_1,z_2)
=\widetilde{F}(z_1/|z_1|,z_2/|z_2|)=\psi(\omega,\theta).
\end{split}
\end{align*}
Moreover,
\begin{align*}
2\pi\omega'&=\Im\frac{z_1'}{z_1}=\Im\frac{iaz_1-F(z_1,z_2)\overline{z_2}}{z_1}=\Im\left(ia-\frac{F(z_1,z_2)\overline{z_1}\overline{z_2}}{|z_1|^2}\right)=a,\\
2\pi\theta'&=\Im\frac{z_2'}{z_2}=\Im\frac{ibz_2-F(z_1,z_2)\overline{z_1}}{z_2}=\Im\left(ib+\frac{F(z_1,z_2)\overline{z_1}\overline{z_2}}{|z_2|^2}\right)=b.
\end{align*}
Hence if $a=2\pi\alpha$ and $b=2\pi$ then in the new coordinates
the differential equation (\ref{rowzesp}) on
$\s^3\setminus(\s^1_+\cup\s^1_-)$ takes the form
\[\omega'=\alpha,\;\;\theta'=1,\;\;s'=\psi\left(\omega,\theta\right),\]
and we return to the scheme from  Section~\ref{difeq} as well.
Therefore if $(z_1,z_2)=\Upsilon(\omega,\theta,s)\in\s^3$ then
$\Psi_t(z_1,z_2)$ equals
\[\left(e^{2\pi i(\omega+t\alpha)}\cos\arctan
e^{s+\int_{0}^t\psi(\omega+\alpha\tau,\theta+\tau)d\tau}, e^{2\pi
i(\theta+t)}\sin\arctan
e^{s+\int_{0}^t\psi(\omega+\alpha\tau,\theta+\tau)d\tau}\right).\]
Recall that $\cos\arctan e^s\to 0$ as $s\to+\infty$  and
$\sin\arctan e^s\to 0$ as $s\to-\infty$.  It follows that if
$\int_{0}^t\psi(\omega+\alpha\tau,\theta+\tau)d\tau\to s_+\infty$
as $t\to+\infty$, then the $\omega$--limit set of $(z_1,z_2)$ is
equal to $\s^1_{s_+}$. Moreover, if
$\int_{0}^t\psi(\omega+\alpha\tau,\theta+\tau)d\tau\to s_-\infty$
as $t\to-\infty$, then the $\alpha$--limit set of $(z_1,z_2)$ is
equal to $\s^1_{s_-}$.

Now as a consequence of results from
Sections~\ref{secpodst}-\ref{secdimension} we have the following
theorem which demonstrates possible coexistence of different
behaviours for solutions of (\ref{rowzesp}).

\begin{Th}
For every irrational $\alpha$ there exists a continuous function
$F:\s^3\to\R$ such that the corresponding flow $(\Psi_t)_{t\in\R}$
on $\s^3$ is transitive and for each $s_-,s_+\in\{-,+\}$ the set
$HC_{s_-s_+}$ of points such that the $\omega$--limit set is equal
to $\s^1_{s_+}$ and the $\alpha$--limit set is equal to
$\s^1_{s_-}$ is dense. If $\alpha\in DC(a)$ for some $a\geq 1$
then for every $0<\gamma<1/((1+a)a)$ the function $F$ can be
chosen $\gamma$-H\"older continuous.

Moreover, for almost every $\alpha$ the Hausdorff dimension of
each set $HC_{s_-s_+}$ is no smaller than $5/2$.
\end{Th}

\appendix
\section{Open problems}
If a continuous function $\va:\T\to\R$ with zero mean has bounded
variation then the uniform Denjoy-Koksma inequality holds, i.e.
$|\va^{(q_n)}(x)|\leq \var \va$ for all $x\in\T$, and as we have
already noticed, this implies the absence of discrete orbits for
$T_\va$. Now assume that $\va:\T\to\R$ is continuous and it
satisfies only $\widehat{\varphi}(n)=\mbox{O}(1/|n|)$. As it has
been proved in \cite{Aa-Le-Ma-Na}, $\va$ fulfills an
$L^2$--Denjoy-Koksma inequality, in particular the sequence
$(\va^{(q_n)})_{n\in\N}$ is bounded in $L^2(\T)$. However, we are
not aware of any direct argument based on the $L^2$--Denjoy-Koksma
inequality which shows the absence of discrete orbits.

\begin{Problem}
Does there exist an irrational rotation $Tx=x+\alpha$ and a
continuous function $\va:\T\to\R$ with $\widehat{\varphi}(n)=
\mbox{O}(1/|n|)$ such that the cylindrical transformation
$T_\varphi$ is Besicovitch?
\end{Problem}

In Section~\ref{holder}, for $\alpha$ satisfying a Diophantine
condition we have constructed a $\gamma$--H\"older continuous
function $\va:\T\to\R$ such that the corresponding cylindrical
transformation is Besicovitch. However $\gamma$ was smaller than
$1/2$.

\begin{Problem}
Does there exist a $\gamma$--H\"older continuous function
$\va:\T\to\R$ with $\gamma\geq 1/2$  such that  $T_\varphi$ is
Besicovitch? Can we construct H\"older continuous Besicovitch
cylindrical transformations  over ``very'' Liouville rotations,
i.e.\ when we assume that the sequence of denominators $(q_n)$
increases very rapidly?
\end{Problem}

We estimated from below the Hausdorff dimension of the sets
$D^{s_-s_+}(\alpha,\va)$ by estimating from below the Hausdorff
dimension of a subset $F^{s_-s_+}$. However, under the condition
(\ref{wzrost}), the Hausdorff dimension of the latter set is not
bigger than $1/2$.

\begin{Problem}
Is $1/2$ an upper bound for the Hausdorff dimension of
$D(\alpha,\va)$ or can we find $\alpha$ and $\va$ so that
$\dim_H(D(\alpha,\va))>1/2$? Is there any relationship between
$\dim_HD(\alpha,\va)$ and the H\"older exponent of $\va$ or the
type of $\alpha$ while estimating the Hausdorff dimension from
above?
\end{Problem}

It would  be also interesting to decide whether there can exist a
Besicovitch transformation in the most extremal sense of the
definition.

\begin{Problem}Does there exist a cylindrical transformation
which is Besicovitch and such that each orbit is either dense or
discrete?
\end{Problem}

There are some natural questions concerning smooth Besicovitch
cylindrical transformations (the problem below being only a
sample).

\begin{Problem}
Does there exist a $C^\infty$ (or even analytic) Besicovitch
cylindrical transformation over a minimal rotation on higher
dimensional tori?
\end{Problem}
\noindent Notice that in the present paper we did not decide
whether a $C^1$ Besicovitch transformation exists over a
two-dimensional minimal rotation.

\begin{Remark}Note that we can view points whose orbits are discrete as
special points which are not recurrent. The results of this paper
can hence be seen as a contribution toward a better understanding
of the problem how big the set of non-recurrent points can be,
both in the compact (as in Section~\ref{sects3}) or non-compact
case.
\end{Remark}

\begin{Remark}
We have not been able to decide whether or not constructions from
Sections~\ref{secpodst}--\ref{highdim} lead to ergodic cocycles.
Notice however that the cocycle $\varphi:\T\to\R$ from
Section~\ref{oduze} is a measurable coboundary. Indeed,
$\|f_k-M_k\|_{L^1}\leq 2q_k\delta_k M_k\leq 2\frac{\log
q_k}{q_{k+1}}$, so the serious $f=\sum_{k=1}^\infty (f_k-M_k)$
converges in $L^1$, and hence $\varphi(x)=f(x+\alpha)-f(x)$ for
a.e.\ $x\in\T$. Moreover, some further modifications of the
construction from Section~\ref{varia} lead also to cocycles which
are measurable coboundaries.
\end{Remark}

\section{Minimality of rotations on tori}
\label{dodatek} Set  $A=16$, $B= \lceil 4^aA\rceil$ and $C= 2B$.
Let $S=\left((A^{(j)}_n,B^{(j)}_n)\right)_{n\geq 1,1\leq j\leq d}$
be a sequence such that $(A^{(j)}_n,B^{(j)}_n)$ is equal to
$(4,C)$ or $(A,B)$. Denote by $\mathfrak{C}(a,S)$ the set of all
$(\alpha_1,\ldots,\alpha_d)\in\T^d$ such that
\begin{equation}\label{nierogol1}
A^{(j)}_n(q_n^{(j-1)})^a\leq q_n^{(j)}\leq
B^{(j)}_n(q_n^{(j-1)})^a\text{ for all }n\geq 1,\;1\leq j\leq d,
\end{equation}
with the notation $q_n^{(0)}=q_{n-1}^{(d)}$. For every subset
$U\subset\{1,\ldots, d\}$ denote by $S^U$ the sequence
$\left((\widetilde{A}^{(j)}_n,\widetilde{B}^{(j)}_n)\right)_{n\geq
1,1\leq j\leq d}$ such
that\[(\widetilde{A}^{(j)}_n,\widetilde{B}^{(j)}_n)=\left\{\begin{array}{ccc}
(4,C)&\text{ if }& j\in U\\
({A}^{(j)}_n,{B}^{(j)}_n)&\text{ if }& j\notin U.
\end{array}\right.\]
Let $\underline{S}$ be the constant sequence with
$(A^{(j)}_n,B^{(j)}_n)=(A,B)$. Then
$\mathfrak{C}(a,\underline{S}^{\{1,\ldots,d\}})=\mathfrak{C}(a)$.

\begin{Lemma}\label{lape1}
The set $\mathfrak{C}(a,\underline{S})\subset\T^d$ uncountable.
\end{Lemma}

\begin{proof}
Let $[0;a_1^{(j)},a_2^{(j)},\ldots]$ stand for the continued
fraction expansion of $\alpha_j$ for $j=1,\ldots,d$. We will use
the notation $a_n^{(0)}:=a_{n-1}^{(d)}$. We will construct
sequences of partial quotients $(a_n^{(j)})_{n=1}^\infty$,
$j=1,\ldots,d$ inductively. In the first step choose any natural
$A\leq a_1^{(1)}\leq B$. Then $q_1^{(1)}=a_1^{(1)}$ fulfills
(\ref{nierogol1}) for $(n,j)=(1,1)$. In the inductive step suppose
that for a pair of natural numbers $(n,l)$ with $1\leq l\leq d$
all partial quotients $a_k^{(j)}$ for $k<n$, $1\leq j\leq d$, and
$a_n^{(j)}$ for $0\leq j<l$ are already chosen so that the
corresponding denominators satisfy the following inequalities
\begin{equation}\label{nierogol}
A(q_k^{(j-1)})^a\leq q_k^{(j)}\leq B(q_k^{(j-1)})^a
\end{equation}
for all pairs $(k,j)$ such that $1\leq k<n$, $1\leq j\leq d$ or
$k=n$, $1\leq j<l$. Now we can choose $a_n^{(l)}\in\N$ so that
\begin{equation}\label{wybora}
A(q_n^{(l-1)})^a\leq a_n^{(l)}q_{n-1}^{(l)}+q_{n-2}^{(l)}\leq
B(q_n^{(l-1)})^a.
\end{equation}
Indeed, by (\ref{nierogol}), $q_{n-2}^{(l)}\leq q_{n-1}^{(l)}\leq
q_n^{(l-1)}/A\leq (q_n^{(l-1)})^a/16$, so we can find at least two
numbers $a_n^{(l)}$ satisfying (\ref{wybora}). Since there are at
least two choices for $a_n^{(l)}$ at each step of the
construction, the set $\mathfrak{C}(a,\underline{S})$ is
uncountable.
\end{proof}

\begin{Lemma}\label{lape2}
Fix $1\leq l\leq d$ and let
$S=\left((A^{(j)}_n,B^{(j)}_n)\right)_{n\geq 1,1\leq j\leq d}$ be
a sequence such that $(A^{(l)}_n,B^{(l)}_n)=(A,B)$ for $n\geq 1$.
If
$\overline{\alpha}=(\alpha_1,\ldots,\alpha_d)\in\mathfrak{C}(a,S)$
then the set
\[\{\alpha\in\T:(\alpha_1,\ldots,\alpha_{l-1},\alpha,
\alpha_{l+1},\ldots,\alpha_d)\in\mathfrak{C}(a,S^{\{l\}})\}\] is
uncountable.
\end{Lemma}
\begin{proof}
By assumption, for each $n\geq 1$
\begin{equation}\label{ogap1}
A(q_n^{(l-1)})^a\leq q_n^{(l)}\leq B(q_n^{(l-1)})^a,
\end{equation}
\[A^{(l+1)}_n(q_n^{(l)})^a\leq q_n^{(l+1)}\leq
B^{(l+1)}_n(q_n^{(l)})^a,\] and hence, by the latter inequality,
\begin{equation}\label{ogap2}
\frac{(q_n^{(l+1)})^{1/a}}{(B^{(l+1)}_n)^{1/a}}\leq q_n^{(l)}\leq
\frac{(q_n^{(l+1)})^{1/a}}{(A^{(l+1)}_n)^{1/a}}.
\end{equation}
Let
\[I_n:=\left[\max\left(4(q_n^{(l-1)})^a,\frac{(q_n^{(l+1)})^{1/a}}
{(B^{(l+1)}_n)^{1/a}}\right),
\min\left(C(q_n^{(l-1)})^a,\frac{(q_n^{(l+1)})^{1/a}}
{(A^{(l+1)}_n)^{1/a}}\right)\right].\] By the definition of
$A,B,C$ and (\ref{ogap1}), (\ref{ogap2}), we have
\begin{eqnarray*}C(q_n^{(l-1)})^a-4(q_n^{(l-1)})^a&\geq
&3q_n^{(l-1)},\\
\frac{(q_n^{(l+1)})^{1/a}}{(A^{(l+1)}_n)^{1/a}}-
\frac{(q_n^{(l+1)})^{1/a}}{(B^{(l+1)}_n)^{1/a}}&=&
\frac{(q_n^{(l+1)})^{1/a}}{(A^{(l+1)}_n)^{1/a}}
\left(1-\frac{(A^{(l+1)}_n)^{1/a}}{(B^{(l+1)}_n)^{1/a}}\right)\\&\geq&
q_n^{(l)}(1-(A/B)^{1/a})\geq \frac{3}{4}q_n^{(l)},\\
\frac{(q_n^{(l+1)})^{1/a}}{(A^{(l+1)}_n)^{1/a}}-4(q_n^{(l-1)})^a&\geq&
q_n^{(l)}(1-4/A)= \frac{3}{4}q_n^{(l)},
\\C(q_n^{(l-1)})^a-\frac{(q_n^{(l+1)})^{1/a}}{(B^{(l+1)}_n)^{1/a}}&\geq&
q_n^{(l)}(C/B-1)= q_n^{(l)}.
\end{eqnarray*}
Since $q_n^{(l-1)}\geq q_{n-1}^{(l+1)}$ and $q_n^{(l)}\geq
4q_{n-1}^{(l+1)}$, it follows that $|I_n|\geq 3q_{n-1}^{(l+1)}$.

Now we construct an irrational number $\alpha=[0;a_1,a_2,\ldots]$
so that
\[(\alpha_1,\ldots,\alpha_{l-1},\alpha,\alpha_{l+1},\ldots,
\alpha_d)\in\mathfrak{C}(a,S^{\{l\}}).\] We construct the sequence
$(a_n)$ of partial quotients of $\alpha$ inductively. Since
$|I_1|\geq 3q_{0}^{(l+1)}=3$, we can choose $a_1\in I_1$. As
$q_1=a_1\in I_1$, we have
\[4(q_1^{(l-1)})^a\leq q_1\leq
C(q_1^{(l-1)})^a,\;A^{(l+1)}_1q_1^a\leq q_1^{(l+1)}\leq
B^{(l+1)}_1q_1^a.\] In the $n$--th step suppose that
$a_1,\ldots,a_{n-1}$ are already selected so that
\[4(q_k^{(l-1)})^a\leq q_k\leq
C(q_k^{(l-1)})^a,\;A^{(l+1)}_kq_k^a\leq q_k^{(l+1)}\leq
B^{(l+1)}_kq_k^a\] for $1\leq k<n$. It follows that
\[|I_n|\geq 3q_{n-1}^{(l+1)}\geq 3q_{n-1}.\]
Therefore there are at least two natural numbers $a_n$ such that
$a_nq_{n-1}+q_{n-2}\in I_n$. Thus
\[4(q_n^{(l-1)})^a\leq q_n\leq
C(q_n^{(l-1)})^a,\;A^{(l+1)}_nq_n^a\leq q_n^{(l+1)}\leq
B^{(l+1)}_nq_n^a.\] It follows that for each $\alpha$ constructed
in this way we have
\[(\alpha_1,\ldots,\alpha_{l-1},\alpha,\alpha_{l+1},
\ldots,\alpha_d)\in\mathfrak{C}(a,S^{\{l\}}).\]
Moreover, there are at least two choices for $a_n$ at each step of
the construction, so the set $\mathfrak{C}(a,S^{\{l\}})$ is
uncountable.
\end{proof}

\begin{Th}
There are uncountably many
$\overline{\alpha}=(\alpha_1,\ldots,\alpha_d)\in\mathfrak{C}(a)$
such that $\alpha_1,\ldots,\alpha_d,1$ are independent over $\Q$.
\end{Th}

\begin{proof}
For each $l=1,\ldots,d$ set
\[\Z^{d,l}=\{\overline{m}=(m_1,\ldots,m_d)\in\Z^d:
m_l\neq 0,m_j=0\text{ for all }l<j\leq d\}.\]
 Let
$(\overline{\alpha}[l])_{l=0}^d$ be a sequence in $\T^d$ defined
inductively as follows:
\begin{itemize}
\item[(i)]$\overline{\alpha}[0]\in\mathfrak{C}(a,\underline{S})$;
\item[(ii)]for each $1\leq l\leq d$ let $\overline{\alpha}[l]$ be an
element of $\mathfrak{C}(a,\underline{S}^{\{1,\ldots,l\}})$ such
that $\overline{\alpha}[l]_j=\overline{\alpha}[l-1]_j$ for all
$j\neq l$ and
$\langle\overline{m},\overline{\alpha}[l]\rangle\notin\Z$ for all
$\overline{m}\in\Z^{d,l}$.
\end{itemize}
The existence of $\overline{\alpha}[0]$ follows directly from
Lemma~\ref{lape1}. The existence of uncountably many
$\overline{\alpha}[l]$ follows from Lemma~\ref{lape2} applied to
$\overline{\alpha}=\overline{\alpha}[l-1]$ and
$S=\underline{S}^{\{1,\ldots,l-1\}}$.

Let
$\overline{\alpha}=(\alpha_1,\ldots,\alpha_d):=\overline{\alpha}[d]$.
Then $\alpha_1,\ldots,\alpha_d,1$ are rationally independent.
Indeed, let $\overline{m}\in\Z^d$ be a nonzero vector and let $l$
be the largest index for which $m_l\neq 0$. Then
$\langle\overline{m},\overline{\alpha}[d]\rangle=\langle\overline{m},\overline{\alpha}[l]\rangle$.
Since $\overline{m}\in\Z^{d,l}$, we conclude that
$\langle\overline{m},\overline{\alpha}\rangle=\langle\overline{m},\overline{\alpha}[d]\rangle=\langle\overline{m},\overline{\alpha}[l]\rangle\notin\Z$.
Moreover,
\[(\alpha_1,\ldots,\alpha_d)=\overline{\alpha}[d]\in\mathfrak{C}(a,\underline{S}^{\{1,\ldots,d\}})=\mathfrak{C}(a).\]
At each step of the construction we had uncountably many choices,
so the  set of constructed vectors $\overline{\alpha}$ is
uncountable as well.
\end{proof}

\section{H\"older continuity}\label{sechol}
Let $f:\T\to\R$ be a continuous function. Set $\D=\{z\in\C:|z|\leq
1\}$. Denote by $\widehat{f}:\D\to\C$ the continuous function
given by $\widehat{f}(z)=zf(\omega)$ whenever $z=|z|e^{2\pi i
\omega}$.
\begin{Lemma}\label{holjed}Suppose $f$ is a $\gamma$-H\"older
continuous function  $(0<\gamma\leq 1)$ with
$|f(\omega)-f(\omega')|\leq M\|\omega-\omega'\|^\gamma$. Then
$\widehat{f}$ is $\gamma$-H\"older continuous and
$|\widehat{f}(z)-\widehat{f}(z')|\leq
2(\|f\|_{C^0}+M)|z-z'|^\gamma$.
\end{Lemma}
\begin{proof}Set $C=\|f\|_{C^0}$. Note
that for all $z=|z|e^{2\pi i\omega}$, $z'=|z'|e^{2\pi i\omega'}$
we have
\begin{align*}|z-z'|^2&=|z|^2+|z'|^2-2\Re
z\overline{z'}=(|z|-|z'|)^2+2|z||z'|\Re(1-e^{2\pi
i(\omega-\omega')})\\&=(|z|-|z'|)^2+4|z||z'|\sin^2(\pi\|\omega-\omega'\|)\geq
4|z||z'|\sin^2(\pi\|\omega-\omega'\|) .
\end{align*}In
view of (\ref{sindul}), it follows that
\begin{equation}\label{podwnier}
|z-z'|\geq 4\sqrt{|z||z'|}\|\omega-\omega'\| .
\end{equation}
Suppose that $|z|\leq|z'|$, then $|z|\leq \sqrt{|z||z'|}$, and
hence
\begin{align*}|\widehat{f}(z)-\widehat{f}(z')|&=|zf(\omega)-z'f(\omega')|\leq
|z-z'||f(\omega')|+|z||f(\omega)-f(\omega')|\\& \leq
C|z-z'|+M\sqrt{|z||z'|}\|\omega-\omega'\|^\gamma.\end{align*}
Since $z,z'\in \D$, we have $|z-z'|\leq
2^{1-\gamma}|z-z'|^{\gamma}$ and
$\sqrt{|z||z'|}\leq\sqrt{|z||z'|}^{\gamma}$, so
\begin{align*}|\widehat{f}(z)-\widehat{f}(z')|&\leq
C2^{1-\gamma}|z-z'|^\gamma+M\sqrt{|z||z'|}^\gamma\|\omega-\omega'\|^\gamma\\&\leq
(2^{1-\gamma}C+M)\left(|z-z'|+\sqrt{|z||z'|}\|\omega-\omega'\|\right)^\gamma.\end{align*}
In view of (\ref{podwnier}), it follows that
\begin{align*}|\widehat{f}(z)-\widehat{f}(z')|\leq (2^{1-\gamma}C+M)2^\gamma |z-z'|^\gamma
\leq 2(C+ M) |z-z'|^\gamma.\end{align*}
\end{proof}
Let $\psi:\T\times\T\to\R$ be a continuous function. Denote by
$F:\D\times\D\to\C$ the function given by
$F(z_1,z_2)=z_1z_2\psi(\omega,\theta)$ whenever $z_1=|z_1|e^{2\pi
i\omega}$, $z_2=|z_2|e^{2\pi i\theta}$.

\begin{Prop}\label{holdwu}
If $\psi$ is $\gamma$-H\"older continuous for some $0<\gamma\leq
1$ then $F$ is $\gamma$-H\"older continuous.
\end{Prop}

\begin{proof}Let $C=\|\psi\|_{C^0}$ and let $M\geq 0$ be  such
that $|\psi(\omega,\theta)-\psi(\omega',\theta')|\leq
M(\|\omega-\omega'\|+\|\theta-\theta'\|)^\gamma$. If
$z_1=|z_1|e^{2\pi i\omega}$, $z_2=|z_2|e^{2\pi i\theta}$,
$z'_1=|z'_1|e^{2\pi i\omega'}$, $z'_2=|z'_2|e^{2\pi i\theta'}$
then, by Lemma~\ref{holjed},
\[|z_1\psi(\omega,\theta')-z'_1\psi(\omega',\theta')|\leq 2(C+M)|z_1-z'_1|^\gamma,\text{ uniformly in }\theta'\]
\[|z_2\psi(\omega,\theta)-z'_2\psi(\omega,\theta')|\leq 2(C+M)|z_2-z'_2|^\gamma,\text{ uniformly in }\omega.\]
Moreover,
\begin{align*}&|F(z_1,z_2)-F(z'_1,z'_2)|=|z_1z_2\psi(\omega,\theta)-z'_1z'_2\psi(\omega',\theta')|
\\&\leq|z'_2||z_1\psi(\omega,\theta')-z'_1\psi(\omega',\theta')|+
|z_1||z_2\psi(\omega,\theta)-z'_2\psi(\omega,\theta')|\\
&\leq|z_1\psi(\omega,\theta')-z'_1\psi(\omega',\theta')|+
|z_2\psi(\omega,\theta)-z'_2\psi(\omega,\theta')|\\
&\leq 2(C+M)\left(|z_1-z'_1|^\gamma+|z_2-z'_2|^\gamma\right)\leq
4(C+M)\left(|z_1-z'_1|+|z_2-z'_2|\right)^\gamma.
\end{align*}
\end{proof}


\begin{thebibliography}{99}
\bibitem{Aa}J.\ Aaronson, {\em An introduction to infinite ergodic theory},
 Mathematical Surveys and Monographs, 50, AMS, Providence, RI, 1997.

\bibitem{Aa-Le-Ma-Na}  J. Aaronson, M.\ Lema\'{n}czyk, C.\ Mauduit,
H.\ Nakada, {\em Koksma's inequality and group extensions of
Kronecker transformations}, in Algorithms, fractals, and dynamics
(Okayama/Kyoto, 1992), 27--50, Plenum, New York, 1995.

\bibitem{At}G. Atkinson, {\em A class of transitive cylinder transformations},
J. London Math. Soc. (2) 17 (1978), 263--270.

\bibitem{Aus}J. Auslander, {\em Minimal flows and their
extensions}, North-Holland Mathematics Studies, 153. Notas de
Matemática, 122. North-Holland Publishing Co., Amsterdam, 1988.

\bibitem{Be}A.S. Besicovitch, {\em A problem on topological transformations of the plane. II},
Proc. Cambridge Philos. Soc. 47 (1951), 38--45.

\bibitem{Ber}N.C. Bernardes, {\em On the set of points with a dense
orbit}, Proc. Amer. Math. Soc. 128 (2000), 3421--3423.

\bibitem{Fal}K. Falconer, {\em Fractal geometry. Mathematical foundations and applications},
John Wiley \& Sons, Inc., Hoboken, NJ, 2003.

\bibitem{Go-He}W.H. Gottschalk, G.A. Hedlund, {\em Topological Dynamics},
Amer. Math. Soc. Colloq. Publ. 36, Amer. Math. Soc., Providence,
RI, 1955.

\bibitem{Her}M.R. Herman, {\em Sur la conjugaison
diff\'erentiable des diff\'eomorphismes du cercle \`a des
rotations}, Publ. Math. IHES {\bf 49} (1979), 5-234.

\bibitem{Kh}A.Ya.\ Khinchin,  {\em Continued fractions}, The University of Chicago
Press, Chicago-London, 1964.

\bibitem{Ku-Ni}L.\ Kuipers, H.\ Niederreiter, {\em Uniform distribution of sequences}, Pure and Applied Mathematics,
Wiley-Interscience, New York-London-Sydney, 1974.

\bibitem{Kw-Si}J.\ Kwiatkowski, A.\ Siemaszko, {\em Discrete orbits for topologically transitive cylindrical
transformations}, Discrete Contin. Dyn. Syst. 27 (2010), 945--961.

\bibitem{LeC-Yo}P. Le Calvez, J-C.\ Yoccoz, {\em Un th\'eor\`eme
d'indice pour les hom\'eomorphismes du plan au voisinage d'un
point fixe}, Annals Math.\ (2) 146 (1997), 241--293.

\bibitem{Le-Me}M. Lema\'nczyk, M.K.\ Mentzen, {\em Topological ergodicity of real cocycles over
minimal rotations}, Monatsh. Math. 134 (2002), 227--246.

\bibitem{Ma-Sh}S. Matsumoto, M. Shishikura, {\em Minimal sets of
certain annular homeomorphisms}, Hiroshima Math. J. 32 (2002),
207--215.

\bibitem{Me-Si}M.K.\ Mentzen, A.\ Siemaszko, {\em Cylinder cocycle extensions of minimal rotations
on monothetic groups}, Colloq.\ Math.\ 101 (2004), 75--88.


\bibitem{Po}H. Poincar\'e, {\em
On curves defined by differential equations} (in Russian),  Ogiz,
Moscow, 1947.

\bibitem{Sch}K.\ Schmidt, {\em Cocycles on ergodic transformation groups},
Macmillan Lectures in Mathematics, Vol. 1. Macmillan Company of
India, Delhi, 1977.

\bibitem{Yo}J.-Ch.\ Yoccoz, {\em Centralisateurs et conjugaison différentiable des difféomorphismes du cercle},
 Petits diviseurs en dimension $1$, Astérisque No. 231 (1995), 89--242.
\end{thebibliography}
\end{document}